\documentclass[12pt]{amsart}
\usepackage{amsmath,amssymb,amsfonts,amsthm,amsopn,url}

\usepackage{graphicx,tikz}
\usepackage[all]{xy}
\usepackage{amsmath}
\usepackage{multirow, longtable, makecell, caption, array}
\usepackage{amssymb}
\usepackage{mathtools}
\usepackage{pb-diagram}
\usepackage{cellspace} %
\newcommand{\ft}{Fourier transform}

\newcommand{\psdo}{pseudodifferential operator}

\newcommand{\tf}{time-frequency}

\newcommand{\aw}{Anti-Wick}

\newtheorem{theorem}{Theorem}[section]
\newtheorem{corollary}[theorem]{Corollary}
\newtheorem{definition}[theorem]{Definition}
\newtheorem{example}[theorem]{Example}
\newtheorem{proposition}[theorem]{Proposition}
\newtheorem{remark}[theorem]{Remark}
\newtheorem{lemma}[theorem]{Lemma}
\newtheorem*{thm*}{Theorem}
\newtheorem*{cor*}{Corollary}

\newcommand{\beqa}{\begin{eqnarray*}}
	\newcommand{\eeqa}{\end{eqnarray*}}
\DeclareMathOperator*{\supp}{supp}

\newcommand{\field}[1]{\mathbb{#1}}
\newcommand{\bR}{\field{R}}
\newcommand{\bN}{\field{N}}
\newcommand{\bZ}{\field{Z}}
\newcommand{\bC}{\field{C}}

\def\al{\alpha}

\def\la{\lambda}
\def\La{\Lambda}

\def\cF{\mathcal{F}}
\def\cS{\mathcal{S}}

\def\cH{\mathcal{H}}
\def\cB{\mathcal{B}}

\def\cG{\mathcal{G}}
\def\cM{\mathcal{M}}

\def\cI{\mathcal{I}}
\def\cC{\mathcal{C}}

\def\a{\aleph}
\def\hf{\hat{f}}

\def\rd{\bR^d}

\def\rdd{{\bR^{2d}}}
\def\zdd{{\bZ^{2d}}}

\def\lrd{L^2(\rd)}

\def\zd{\bZ^d}

\def\intrd{\int_{\rd}}
\def\intrdd{\int_{\rdd}}

\def\<{\left<}
\def\>{\right>}

\def\mv1{M_v^1}

\def\mpq{M^{p,q}}
\def\Mmpq{M_w^{p,q}}
\def\phas{(x,\omega )}

\newcommand{\abs}[1]{\lvert#1\rvert}
\newcommand{\norm}[1]{\lVert#1\rVert}
\def\o
\def\vu{\emptyset}

\def\o{\omega}
\def\a{\alpha}
\def\b{\beta}

\def\Ren{\mathbb{R}^d}
\def\Renn{\mathbb{R}^{2d}}

\def\Fur{\mathcal{F}}

\def\f{g}

\def\gaw{A_a^{\f_1,\f_2}}

\def\Sn2{S_{2}(L^{2}(\Ren))}
\def\S1{S_{1}(L^{2}(\Ren))}
\def\sig00{\sigma_{0,0}}

\def\la{\langle}
\def\ra{\rangle}

\newcommand{\aaf}{A_a^{\f_1,\f_2}}

\def\sfG{\mathsf{G}}
\def\sfLa{\mathsf{\Lambda}}
\def\dFou{\mathcal{F}_N}
\def\dGab{\sfG^{g_1,g_2}_{a}}
\newcommand{\dGabor}[4]{\sfG^{#1,#2}_{#3,#4}}
\usepackage{mathrsfs}
\def\scrA{\mathscr{A}}
\def\sfF{\mathsf{F}}
\def\sfA{\mathsf{S}}
\def\sfP{\mathsf{P}}
\def\AP{\sfA_\sfP^{BA}}

\usepackage[OT2,T1]{fontenc}
\DeclareSymbolFont{cyrletters}{OT2}{wncyr}{m}{n}
\DeclareMathSymbol{\Sha}{\mathalpha}{cyrletters}{"58}

\begin{document}
	\begin{abstract} We study the connection between STFT multipliers $A^{\f_1,\f_2}_{1\otimes m}$ having windows $\f_1,\f_2$, symbols $a\phas=(1\otimes m)\phas=m(\o)$, $\phas\in\rdd$, and the Fourier multipliers $T_{m_2}$ with  symbol $m_2$ on $\rd$. We find sufficient and necessary conditions on symbols $m,m_2$ and windows $\f_1,\f_2$ for the equality $T_{m_2}= A^{\f_1,\f_2}_{1\otimes m}$.
		For $m=m_2$ the former equality 	 holds only for particular choices of window functions in modulation spaces, whereas it never occurs in the realm of  Lebesgue spaces. In general, the STFT multiplier $A^{\f_1,\f_2}_{1\otimes m}$, also called localization operator, presents a smoothing effect due to the so-called \emph{two-window short-time Fourier transform} which enters in the definition of $A^{\f_1,\f_2}_{1\otimes m}$. As a by-product we prove necessary conditions for the continuity of  anti-Wick operators $A^{\f,\f}_{1\otimes m}: L^p\to L^q$  having multiplier $m$ in weak $L^r$ spaces. Finally, we exhibit the related results for their discrete counterpart: in this setting STFT multipliers are called Gabor multipliers whereas Fourier multiplier are better known as linear time invariant (LTI) filters.
	\end{abstract}
	
	\title[Comparisons between Fourier and STFT multipliers]{Comparisons between Fourier and STFT multipliers: the smoothing effect of the Short-time Fourier Transform}
	\author[\tiny P.Balazs]{Peter Balazs}
	\address{Acoustics Research Institute}
	%\curraddr{\"{y} }
	\email{peter.balazs@oeaw.ac.at}
	\thanks{}
	\author[F. Bastianoni]{Federico Bastianoni}
	\address{Depatment of Mathematical Sciences, Politecnico di Torino, corso Duca degli Abruzzi 24, 10129 Torino, Italy}
	%\curraddr{\"{y} }
	\email{federico.bastianoni@polito.it}
	\thanks{}
	\author[E. Cordero]{Elena Cordero}
	\address{Dipartimento di Matematica, Universit\`a di Torino, via Carlo Alberto 10, 10123 Torino, Italy}
	%\curraddr{\"{y} }
	\email{elena.cordero@unito.it}
	\thanks{}

	\author[H.G. Feichtinger]{Hans G. Feichtinger}
	\address{Department of Mathematics, University of Vienna, Vienna, Austria}
	%\curraddr{\"{y} }
		\email{hans.feichtinger@univie.ac.at}
	\thanks{}
		\author[N. Schweighofer]{Nina Schweighofer}
	\address{European Central Bank, Taunustor 2, 60311, Frankfurt am Main, Germany}
	%\curraddr{\"{y} }
	\email{nina.schweighofer@ecb.europa.eu}
	\thanks{}
	
	\date{}
	\maketitle
\section{Introduction}

STFT multipliers, also called localization operators,  have been introduced  by Daubechies \cite{Daube88} and investigated by Ramanathan and Topiwala~\cite{RT93} as a mathematical tool to  restrict functions to a region in the  \tf\ plane and  to extract \tf\ features. Special instances of localization  operators are  the so-called ``Anti-Wick operators'', introduced earlier by Berezin~\cite{Berezin71} in the framework of quantum mechanics, details can be found in Shubin's book \cite{Shubin91}. For this reason, they have been
widely studied in signal
analysis and other applications~\cite{Fei-Now02,WongLocalization}. Their \emph{discrete versions} are known as Gabor multipliers \cite{Fei-Now02}. Motivated by the overlap-add implementation of convolution \cite{oppenheim}, in some signal processing software system linear filtering is implemented by using Gabor multipliers  instead of working out a convolution in time domain like in e.g. STx \cite{badelano00,Zala2017}. There are several reasons for doing that. First, they are easy
to implement. Then, quasi real time processing with finite support windows is possible so only a short delay will be produced. %Moreover, in order to implement a filter in a linear setting it is straight forward to do this by masking unwanted frequency components. 
Moreover, the conceptual relation is clear, and the interpretation of a bandpass filter, i.e. masking unwanted frequency components is straight forward.
In many chosen settings the results also closely match the expectations (see Example \ref{sec:firstexampl0}). 
The problem of representation and approximation of linear operators by means
of Gabor multipliers (and suitable modifications) was studied by D{\"o}rfler and Torr\'{e}sani in \cite{DT2010}, further investigations are contained in \cite{Gibson2013,Toetalii}. 
More generally,  approximating  problems for \psdo s via STFT multipliers (``wave packets'' were exhibited in the work by Cordoba and Fefferman \cite{CF78}, see also Folland \cite{folland89} and the PhD thesis \cite{Nina}). In this paper, in contrast, we investigate under which conditions is it possible to write a filter as a Gabor multiplier exactly.

%\textcolor{red}{Finally, we should  definitely mention time-varying filter (with smooth symbol) (Peter)}.

To motivate our result let us introduce the STFT multipliers by a time-frequency representation, the  short-time Fourier transform (STFT), as follows. First, recall    the  modulation $M_{\omega}$ and translation $T_{x}$ operators of a function $f$ on $\rd$:
\[
M_{\omega}f\left(t\right)= e^{2\pi it \omega}f\left(t\right),\qquad T_{x}f\left(t\right)= f\left(t-x\right),\quad \omega,x\in\rd.
\]
%For $z=(z_1,z_2)\in\rdd$, we define the \tf \, shift $\pi(z)=M_{z_2} T_{z_1}$. 
Fix a non-zero Schwartz function $\f\in\cS(\rd)\setminus\{0\}$.  We define the short-time Fourier transform of a tempered distribution $f\in\cS'(\rd)$ as
\begin{equation}\label{STFTdef}
V_{\f}f\phas=\la f, M_\o T_x\f\ra=\Fur (f\cdot\overline{T_x \f})(\omega)=\int_{\Ren}
f(y)\, {\overline {\f(y-x)}} \, e^{-2\pi iy \o }\,dy.
\end{equation}

The STFT multiplier $\aaf $ with symbol $a$, analysis window  $\f _1$, and synthesis window $\f _2$ is formally defined   to be
\begin{equation}
\label{eqi4}
\aaf f(t)=\int_{\Renn}a \phas V_{\f _1}f \phas M_\omega T_x \f _2 (t)
\,
dx d\omega.
\end{equation}
If $\f _1(t)  = \f _2 (t) = e^{-\pi t^2}$, then the STFT multiplier  $A_a = \aaf $ becomes  the
classical \aw\ operator. We recall that the mapping $a \mapsto \aaf  $ is  a quantization rule \cite{Berezin71,deGossonsymplectic2011,Shubin91,WongLocalization}.

For $\a,\beta>0$, consider the lattice $\Lambda=\a\zd\times \beta\zd$, then a Gabor multiplier 
with windows $\f_1,\f_2\in L^2(\rd)$ can formally be defined as
\begin{equation}\label{GB}
G_{a}^{\f_1,\f_2}f=\sum_{k,n\in\zdd} a(\a k,\beta n) V_{\f_1}f(\a k,\beta n) T_{\a k}M_{\beta n}\f_2,\quad f\in \lrd,
\end{equation}
 see also \cite{Fei-Now02,sc04-2}. Observe that a Gabor multiplier is the \emph{discrete} version of a STFT multiplier; in fact it can be obtained  from \eqref{eqi4} by replacing the Lebesgue measure $dx d\o$ with the discrete measure $\nu= \sum_{k,n\in\zd} \delta_{\a k,\beta n}$; the integration with respect to $\nu$ becomes the summation $$\intrdd F\phas d\nu(x,\o) = \sum_{k,n\in\zd} F(\a k,\beta n).$$
%Hence the main results for STFT multipliers (cf. Theorem \ref{main}) can be easily adapted to the framework of Gabor multipliers. 
Note that this is a particular instance of a continuous frame multiplier, a (discrete) frame multiplier and their relation, see \cite{xllmult1, BalBayRah2012, balsto09new}.
This is also true for the next concept, Fourier multipliers.

Fourier multipliers \cite{BenGraGroOko2005}, also named linear time invariant (LTI) filters \cite{OppSch99}, are well known in both  partial differential equations and signal analysis. %They can be viewed as a special instance of Kohn-Nirenberg (KN) operators $T_m$ with KN symbol $m$ which depends only on the frequency variables $\o\in\rd$. 
Precisely, a Fourier multiplier with \emph{multiplier} $m\in\cS'(\rd)$ is defined by
\begin{equation}\label{FM}
T_mf(t)=\cF^{-1}(m\cF f)(t)=\left(\cF^{-1}m\ast f\right)(t),		\quad f\in\cS(\rd).
\end{equation}
In Section \ref{6} we shall adopt the notation 
\begin{equation}\label{Eq-h}
	h=\cF^{-1}m
\end{equation}
both for the continuous and finite discrete setting. This is called the \emph{transfer function} in signal processing \cite{oppenheim}.\\
Such operator is a well-defined linear mapping from $\cS(\rd)$ to $\cS'(\rd)$.
Boundedness properties of Fourier multipliers  $T_m : L^p(\rd)\to L^q(\rd)$  are studied in the classical paper by H\"{o}rmander \cite{Hormander1960}.
The most important examples of Fourier multipliers can be obtained by taking $p=q=2$. Then $T_m$ is bounded if and only if the multiplier $m\in L^\infty(\rd)$ and $\|T_m\|_{B(L^2)} =\|m\|_\infty$.
For $p=q=1$ and $p=q=\infty$   the only bounded Fourier multipliers are Fourier transforms of bounded measures. For the  cases $p=q\in (1,\infty)\setminus\{2\}$ only sufficient conditions on $m$ are known. The assumptions $m\in L^\infty$ is necessary, though. The main result by  H\"{o}rmander in  \cite[Theorem 1.11]{Hormander1960}  (see also its generalization to locally compact groups \cite{Ruzh2020}) states:
\begin{theorem}\label{hormander}
 If $1<p\leq 2\leq q<\infty$, $m\in L^{r,\infty}(\rd)$ with \begin{equation}\label{hormindices}
1/q=1/r+1/p,
 \end{equation} then $T_m$ is bounded $T_m: L^p(\rd)\to L^q(\rd)$.  
\end{theorem} 
Here $L^{r,\infty}(\rd)$ is the weak $L^r$-space, see \eqref{lrinfty}  in the Preliminaries below.   For example, every $m$ on $\rd$ with $|m(\o)|\leq C|\o|^{-d/r}$, $C>0$, satisfies $m\in L^{r,\infty}(\rd)$. For simplicity, we define $L^{\infty,\infty}(\rd):=L^\infty(\rd)$, so that, inserting $r=\infty$ in Theorem \ref{hormander}  we recapture the boundedness of the multiplier $T_m$ on $\lrd$.
\par
In signal analysis, both Gabor and Fourier multipliers are introduced  to \emph{localize} a signal; the former in the time-frequency space, the latter only in the frequency space.   Since, as we said above, Gabor multipliers are easy to be implemented numerically, 
the mathematical question is to determine under which conditions a Gabor multiplier  is equivalent to a linear time invariant filter. More generally, we aim at answering the following question: 

\vspace{0.1truecm}
\emph{Given a STFT multiplier $\gaw$ with symbol $a\phas=(1\otimes m)\phas=m(\o)$, $x,\o\in\rd$ ($a$ depends only on \emph{the frequency variable} $\o\in\rd$),  is it possible to write it in the form of a Fourier multiplier?}\par 
We study the equality
\begin{equation}\label{GPtilde}
A^{\f_1,\f_2}_{1\otimes m}=T_{m_2}\quad\mbox{on}\quad \cS(\rd).
\end{equation}
 
In order to give a flavour of our results, we need to introduce a function which correlates $\f_1$ and $\f_2$.
Recall the reflection operator $\cI$ of a function $f$ on $\rd$
\begin{equation}\label{refop}
\cI f(t)=f(-t),\quad t\in\rd.
\end{equation}
For $\f_1,\f_2\in\lrd$,  the \emph{window correlation function}   of the pair $(\f_1,\f_2)$ is defined by
\begin{equation}\label{tildeG}
\cC_{\f_1,\f_2}(y)=(\cI \f_2\ast  \bar{\f_1})(y),\quad y\in\rd
\end{equation}
(observe that ${\cC}_{\f_1,\f_2}$ is a continuous function on $\rd$). The window correlation function enjoys several properties depending on the function/distribution space of the windows $\f_1,\f_2$, cf. Proposition \ref{propwindow} in the sequel.

The equality \eqref{GPtilde} is  possible if and only if
\begin{equation}\label{fab}
m_2=m\ast \cF^{-1}(\cC_{\f_1,\f_2}),
\end{equation}
with $ m,m_2\in \cS'(\rd)$, $\f_1,\f_2\in\cS(\rd)$ or other suitable function spaces, see Theorem \ref{main}.

In particular, if we choose $m=m_2$   the equality \eqref{GPtilde} holds for any multiplier $m\in \cS(\rd)$ if and only if
\begin{equation}\label{G-uno}
{\cC}_{\f_1,\f_2}=1 \quad \mbox{in} \quad \cS'(\rd).
\end{equation}

Condition \eqref{G-uno} above is very restrictive, so that \eqref{GPtilde} never holds for classical anti-Wick operators, whose Gaussian windows  provide a smoothing effect we shall explain presently.

First, we recall that the H\"{o}rmander's condition $p\leq 2\leq q$ in Theorem \ref{hormander} is sharp. More precisely, if there exists a function $F$ such that $\{F>0\}$ has non-zero measure and for all $m:\rd \to\bR$ with $|m|\leq |F|$, $T_m : L^p(\rd) \to L^q (\rd)$ is bounded, then $p\leq 2\leq q$ (cf. \cite[Theorem 1.12]{Hormander1960}). Moreover, also \eqref{hormindices} is necessary by the $L^p$ inequalities for potentials (see \cite[pag. 119]{stein}).  We shall present a direct proof by rescaling arguments of the following  necessary condition (see Section \ref{Sec-4}): 
\begin{proposition}\label{hormandernec}
	For $p,q,r\in (1,\infty]$ we assume that the Fourier multiplier $T_m$ satisfies
	\begin{equation}\label{hormest}
	\|T_m f\|_q\leq C \|m\|_{L^{r,\infty}}\|f\|_p, \quad\mbox{for\,every\,}\,f,\,m\in \cS(\rd),
	\end{equation} 
	then we must have the indices' relation:
	\begin{equation}\label{suffcondanti-Wick}\frac1q\leq \frac1r+\frac1p.\end{equation}
\end{proposition}

In this paper we also investigate the smoothing effects of the anti-Wick operator  $A^{\f,\f}_{1\otimes m}$ with respect to the corresponding  Fourier multiplier $T_m$.  It can be stated as follows (see the proof in Section \ref{4}). Please note the similarity (and differences) to H\"{o}rmander's result, Theorem \ref{hormander}).
\begin{theorem}\label{contpqwick}
	If $1<p\leq 2\leq q<\infty$,  $m\in L^{r,\infty}(\rd)$ with indices satisfying \eqref{suffcondanti-Wick}, 
	then the anti-Wick operator $A^{\f,\f}_{1\otimes m}$ is bounded from $L^p(\rd)$ into $L^q(\rd)$.
\end{theorem}

The previous result holds true for more general STFT multipliers $A^{\f_1,\f_2}_{1\otimes m}$ with $\f_1,\f_2\in\cS'(\rd)$ such that the window correlation function satisfies ${\cC}_{\f_1,\f_2}\in L^2(\rd)\cap L^\infty(\rd)$, cf. Theorem \ref{contpq} in Section \ref{4} below.
For $p=2$, the boundedness of the Fourier multiplier $T_m$ in Theorem \ref{hormander} forces the indices' choice: $q=2$ and $r=\infty$, whereas condition in \eqref{suffcondanti-Wick} is more flexible, allowing to choose $q\geq 2$ and  $r\leq\infty$.

The necessity of condition \eqref{suffcondanti-Wick} for anti-Wick operators is proved in Theorem \ref{neccWick}.

We conjecture  that other possible smoothing effects could be shown by replacing $L^p$ and $L^{r,\infty}$ with Wiener amalgam spaces (cf. \cite{feichtinger-modulation}). This will be the subject of future investigations.
%For more general results concerning smoothing effects of localization operators we refer to Section \ref{4}.

%If the exact equivalence in \eqref{GPtilde} is not possible, the interest for a best approximation appears naturally.  This will be the subject of  Section \ref{6}, where  error estimates as a function of lattice parameters  and the synthesis window used are derived.
%Finally, a so-called \emph{simple estimation approach}  will give an upper bound on the difference between the LTI filter and the Gabor multiplier and study the effectively implemented operator.
The connection between Fourier and Gabor multipliers was studied earlier by Weisz in \cite{Weisz2008}. The focus is different and  can be viewed in our framework as follows: if a symbol $m$ gives rise to a Fourier multiplier $T_m$ which is bounded on $L^p(\rd)$, then the STFT multiplier $A_{1\otimes m}^{\f_1,\f_2}$ is also bounded on $L^p$ (and more generally, on Wiener amalgam spaces cf. \cite[Theorem 8]{Weisz2008}), provided the windows are \emph{smooth} enough, that is, are included in suitable Wiener amalgam spaces containing the modulation space $M^1$.
\par 
For applications, we will study the finite dimensional discrete setting, considering signals $f \in \bC^N$ in Section \ref{6}, which is an extension of \cite[Chapter 2]{Nina}. The problems under investigation are similar to the ones for the continuous setting, the tools at hand however are sometimes different. 
%Where necessary, for a better understanding, we will also investigate an in-between step and prove results for the infinite-dimensional discrete setting, i.e. $\ell^2$.

%As a by-product, we obtain new boundedness results for localization operators $\gaw$ with windows $\f_1,\f_2$ in modulation spaces, cf. Theorem \ref{class}.

%Elena: {\color{red} Idee ulteriori per studiare le connessioni? O altro? }
\section{Preliminaries}
{\bf Notations.} In this paper $\hookrightarrow$ denotes the continuous embeddings of function spaces. The conjugate exponent $p^{\prime }$ of $p\in \lbrack
1,\infty ]$ is defined by $1/p+1/p^{\prime }=1$.

 For $r\in [1,\infty)$, the \emph{ weak $L^r$ space}  $L^{r,\infty}(\rd)$ is the space of measurable functions $f:\rd\to\bC$ such that 
\begin{equation}\label{lrinfty}
\|f\|_{L^{r,\infty}}:= \sup_{\a>0} \a\lambda_f(\a)^{\frac1r} <\infty,
\end{equation}
where $\lambda_f(\a):=\mu(\{t\in\rd: |f(t)|>\a \})$, $\a>0$,  with $\mu$ being the Lebesgue measure (see, e.g., \cite{triebel2010theory}).  

Note that  the quantity in \eqref{lrinfty} is only a quasi-norm.   \par 
For convenience,  we  write $L^{\infty,\infty}(\rd):= L^\infty(\rd)$.  Observe that weak $L^r$ spaces are special instances of \emph{Lorentz spaces}  and  $L^r(\rd)\subseteq L^{r,\infty}(\rd)$, $1\leq r\leq\infty$.

 For $t=(t_1,\ldots,t_d), \o=(\o_1,\ldots,\o_d)\in \rd,$ the inner product is denoted by $t\o=t\cdot \o=t_1\o_1+\ldots+ t_d\o_d$. So that we adopt the notation $t^2=\abs{t}^2=t^1_1+\ldots+t^2_d$. For $f\in L^1(\rd)$ the \ft \, is normalized to be $$\cF f(\o)=\hat{f}(\o)=\intrd e^{-2\pi i t\o} f(t) \,dt.$$

{\bf Weight functions.}
We denote by $v$  a
continuous, positive, even,  submultiplicative  weight function on $\rd$, i.e., 
$ v(z_1+z_2)\leq v(z_1)v(z_2)$, for all $ z_1,z_2\in\Ren$.
We say that $w\in \mathcal{M}_v(\rd)$  if $w$ is a positive, continuous, even  weight function  on $\Ren$  which is
	$v$-moderate, i.e.
$ w(z_1+z_2)\leq Cv(z_1)w(z_2)$,  for all $z_1,z_2\in\Ren$ and some $C>0$.
We will mainly work with polynomial weights of the type
\begin{equation}\label{vs}
v_s(z)=\la z\ra^s =(1+|z|^2)^{s/2},\quad s\in\bR
\end{equation}
(for $s<0$, $v_s$ is $v_{|s|}$-moderate).\par 
Given two weight functions $w_1,w_2$ on $\rd$, we write $$(w_1\otimes w_2)(x,\o):=w_1(x)w_2(\o),\quad x,\o\in \rd.$$

{\bf Modulation spaces.}  These spaces were introduced by the author of \cite{feichtinger-modulation}, where many of their properties were already investigated. Nowadays, they are treated in many textbooks, see, e.g. \cite{beok20,Elena-book}. For a general extension to the quasi-Banach setting on locally compact Abelian groups we mention the recent \cite{BasCor2021}.\par 
	Fix a non-zero window $g$ in the Schwartz class $\cS(\rd)$, a weight $w\in\mathcal{M}_v$ and $1\leq p,q\leq \infty$. The modulation space $M^{p,q}_w(\rd)$ consists of all tempered distributions $f\in\cS'(\rd)$ such that the norm 
	\begin{equation}\label{norm-mod}
	\|f\|_{M^{p,q}_w}=\|V_gf\|_{L^{p,q}_w}=\left(\intrd\left(\intrd |V_g f \phas|^p w\phas^p dx  \right)^{\frac qp}d\o\right)^\frac1q 
	\end{equation}
	(natural changes with $p=\infty$ or $q=\infty)$ is finite. 
 If $p=q$, we
write $M^p_w(\rd)$ instead of $M^{p,p}_w(\rd)$; if $w\equiv1$, we write $M^{p,q}(\rd)$ in place of $M^{p,q}_1(\rd)$.

The space $M^{p,q}(\mathbb{R}^d)$ is a Banach space whose definition is
independent of the choice of the window $g\in\cS(\rd)$,  that is, different
non-zero window functions in the Schwartz class yield equivalent norms. Furthermore, the window class can be extended to the modulation space $M^1(\rd)$, also known as \emph{Feichtinger's algebra}, see \cite{ja18}. The modulation space $%
M^{\infty,1}(\rd)$ is also called \emph{Sj\"ostrand's class} \cite{Sjostrand1}.
%
%The closure of $\mathcal{S}(\mathbb{R}^d)$ in the $M^{p,q}$-norm is denoted $%
%\mathcal{M}^{p,q}(\mathbb{R}^d)$. Then 
%\begin{equation*}
%\mathcal{M}^{p,q}(\mathbb{R}^d) \subseteq M^{p,q}(\mathbb{R}^d), \quad {%
%	\mbox and }\,\, \mathcal{M}^{p,q} (\mathbb{R}^d) = M^{p,q} (\mathbb{R}^d),
%\end{equation*}
%provided $p<\infty$ and $q<\infty$.

For any $p,q\in \lbrack
1,\infty ]$, the inner product $\langle \cdot {,}\cdot \rangle_{L^2(\rd)}=\langle \cdot {,}\cdot \rangle$ restricted to $\cS
(\mathbb{R}^{d})\times \cS(\mathbb{R}^{d})$ extends to a continuous
sesquilinear map $M^{p,q}(\mathbb{R}^{d})\times M^{p^{\prime },q^{\prime }}(%
\mathbb{R}^{d})\rightarrow \mathbb{C}$.

 For $1\leq p,q<\infty$, the duality property for $\mpq(\rd)$  is given by 
$$(\Mmpq)'(\rd)= M^{p',q'}_{1/w}(\rd),$$
with $p',q'$ being the conjugate exponents and
\begin{equation}\label{duality}
\la f, g\ra =\intrdd V_h f(z) \overline{ V_h g}(z)dz,\quad f\in\Mmpq(\rd), g\in M^{p',q'}_{1/w}(\rd),
\end{equation}
$w\in\mathcal{M}_v$, for any fixed $h\in M^1_v(\rd)\setminus\{0\}$.
Observe that  H\"{o}lder's inequality  for $L^{p,q}_m$ spaces let us write, for every $1\leq p,q\leq\infty$, 
\begin{equation}\label{Holdermodulation}
|\la f, g\ra| \leq \|f\|_{\Mmpq}\|g\|_{M^{p',q'}_{1/w}}, \quad f\in\Mmpq(\rd), g\in M^{p',q'}_{1/w} (\rd).
\end{equation}
We denote by $\mathcal{M}^{p,q}_w(\rd)$
the closure of $\cS(\rd)$ in the  $M^{p,q}_w$-norm. Observe
that $\mathcal{M}^{p,q}_w(\rd)=M^{p,q}_w(\rd)$, whenever the indices $p$ and
$q$ are finite. 
Notice that these spaces  enjoy the duality property
$(\mathcal{M}^{p,q}_w)'=\mathcal{M}^{p',q'}_{1/w}$, with $1\leq
p,q\leq\infty$.
Modulation spaces satisfy the following inclusion properties: 
\begin{equation}
\mathcal{S}(\mathbb{R}^{d})\hookrightarrow M^{p_{1},q_{1}}_w(\mathbb{R}%
^{d})\hookrightarrow M^{p_{2},q_{2}}_w(\mathbb{R}^{d})\hookrightarrow \mathcal{S}^{\prime
}(\mathbb{R}^{d}),\quad p_{1}\leq p_{2},\,\,q_{1}\leq q_{2},
\label{modspaceincl1}
\end{equation}%
  Moreover, $\mathcal{S}(\mathbb{R}^{d})$ is dense in $M^{p,q}_w(\mathbb{R}^d)$ whenever $p<\infty$ and $q<\infty$.
In what follows, we recall the convolution properties for modulation spaces, cf.  \cite[Proposition 2.4.19]{Elena-book}.
\begin{proposition}\label{propconv} 
Let $\nu (\omega )>0$ be  an  even weight function  on $\Ren$. Furthermore let  $1\leq
p,q,r,t,u,\gamma\leq\infty$, with
\begin{equation}\label{Holderindices}
\frac 1u+\frac 1t\geq \frac 1\gamma,
\end{equation}
and 
\begin{equation}\label{Youngindicesrbig}
\frac1p+\frac1q\geq 1+\frac1r.
\end{equation}
For given $w\in\mathcal{M}_v(\rdd)$,   let $w_1$ and $w_2$ be the restriction to $\Ren\times\{0\}$ and  $\{0\}\times\Ren$ respectively, i.e $w_1(x)
\coloneqq w(x,0) $ and $w_2(\omega ) \coloneqq w(0,\omega )$. Define $v_1$ and $v_2$ in a similar way. If $f\in M^{p,u}_{w_1\otimes \nu}(\Ren)$, $h\in M^{q,t}_{v_1\otimes
	v_2\nu^{-1}}(\Ren)$ then  $f\ast h\in M^{r,\gamma}_w(\Ren)$
with  norm inequality  \begin{equation}\label{mconvm}\| f\ast h \|_{M^{r,\gamma}_w}\lesssim
\|f\|_{M^{p,u}_{w_1\otimes \nu}}\|h\|_{ M^{q,t}_{v_1\otimes
		v_2\nu^{-1}}}.\end{equation}
\end{proposition}
\begin{proposition}\label{2.4}
	Consider  $1\leq p,q\leq 	\infty$, with $p',q' $ being conjugate exponents of $p,q$, respectively. \\
	  (i) For $1\leq p,q\leq \infty$,  $f\in\mathcal{M}^{p,q}(\rd)$, $h\in \mathcal{M}^{p',q'}(\rd)$, we have that $f\ast h\in \cC_0(\rd)$.\\
	  (ii) For $1<p,q<\infty$,  $f\in\mpq(\rd)$, $h\in M^{p',q'}(\rd)$, we have that $f\ast h\in \cC_0(\rd)$. \\
	   (iii) If either  $f\in M^{\infty,1}(\rd)$ and $h\in M^{1,\infty}(\rd)$ or $f\in M^1(\rd)$ and $h\in M^\infty(\rd)$, then $f\ast h\in \cC_b(\rd)$. 
\end{proposition}
\begin{proof} These results are well known, see \cite{fei1981} and \cite{fei1983}. For sake of clarity we provide a direct proof.\par
$(i)$  Using the density of $\cS(\rd)$ in both spaces we can find sequences $\{f_n\},\{h_n\}\in\cS(\rd)$ such that $\|f_n-f\|_{M^{p,q}}\to 0$ and $\|h_n-h\|_{M^{p',q'}}\to 0$, now $f_n\ast h_n\in\cS(\rd)\hookrightarrow \cC_0(\rd)$ so that, using $$|f\ast h(t)|=|\la f,\overline{ T_t\cI(h)}\ra|\leq \|f\|_{\mpq}\|\overline{T_t\cI(h)}\|_{M^{p',q'}}=\|f\|_{\mpq}\|h\|_{M^{p',q'}},\quad \forall t\in\rd,$$
\begin{align*}\|f_n\ast h_n-f\ast h\|_{L^\infty}&\leq \|f_n\ast(h_n-h)\|_{L^\infty}+\|(f_n-f)\ast h\|_{L^\infty}\\&\leq \|f_n\|_{\mpq}\|h_n-h\|_{M^{p',q'}}+\|f_n-f\|_{\mpq}\|h\|_{M^{p',q'}}.
\end{align*}
Hence $f\ast h\in\cC_0(\rd)$. Item $(ii)$ is obtained by the same argument as in $(i)$.\par
$(iii)$ Using the convolution relations of Proposition \ref{propconv} we infer
$$M^{\infty,1}(\rd)\ast M^{1,\infty}(\rd)\hookrightarrow M^{\infty,1}(\rd)\,\,\mbox{and}\,\,\, M^{1}(\rd)\ast M^{\infty}(\rd)\hookrightarrow M^{\infty,1}(\rd).$$  
It follows immediately from the definition of the modulation space $M^{\infty,1}(\rd)$ that
\begin{equation}\label{minfty1cont}M^{\infty,1}(\rd)\subset (\cF L^1(\rd))_{loc}\cap L^\infty(\rd)\subset \cC_b(\rd)
\end{equation}
and we are done. 
\end{proof}
\begin{remark}
We observe that the convolution relations
$$M^1(\rd)\ast M^\infty(\rd)\subset \cC_b(\rd)$$
	where already  shown in \cite[Lemma 8]{FN2005}. 
\end{remark}

Here we show an optimal result for $\mpq$-boundedness (and in particular $L^2$-boundedness) of STFT multipliers. We extend Theorem 5.2 in \cite{Wignersharp2018} and Theorem 1.1 in \cite{EleCharly2003}. 

\begin{theorem}\label{class}
	Consider $s\geq0$,  $p_1,p_2,q_1,q_2\in [1,\infty]$, with $1/p_1+1/p_2\geq 1$, $1/q_1+1/q_2\geq 1$. If $\f_1\in M^{p_1,q_1}_{v_s}(\rd)$, $\f_2\in M^{p_2,q_2}_{v_{s}}(\rd)$, and  $a\in M^{\infty,1}(\rdd)$,
	then	$\aaf $ is  bounded  on every $\mpq_{v_s}(\rd)$, $p,q\in [1,\infty]$. In particular, the operator $\aaf$ is bounded on the Shubin-Sobolev space $\mathcal{Q}_s:=M^2_{v_s}$ (In particular, for $s=0$, $\aaf $ is bounded on $L^2(\rd)$).
\end{theorem}
\begin{proof}
	If $\f_1\in M^{p_1,q_1}_{v_s}(\rd)$, $\f_2\in M^{p_2,q_2}_{v_{s}}(\rd)$ with $1/p_1+1/p_2\geq 1$, $1/q_1+1/q_2\geq 1$, by  \cite[Theorem 4]{E-Wigner-QBanach} we infer that their cross-Wigner distribution 
	\begin{equation*}
		W(\f_2,\f_1)\phas=\intrd \f_2\left(x+\frac t2\right)\overline{\f_1\left(x-\frac t2\right)}e^{-2\pi i t\o}\,dt
	\end{equation*}
	is in $M^{1,\infty}_{1\otimes v_s}(\rdd)$. Rewriting the STFT multiplier $\gaw$ as a Weyl operator $L_\sigma$ with $\sigma=a\ast W(\f_2,\f_1)$, the convolution relations for modulation spaces in Proposition \ref{propconv} give
	$$\sigma\in M^{\infty,1}(\rdd)\ast M^{1,\infty}_{1\otimes v_s}(\rdd)\hookrightarrow M^{\infty,1}_{1\otimes v_s}(\rdd). $$
	The result follows by the continuity properties of Weyl operators in  \cite[Theorem 5.2]{Wignersharp2018}.
\end{proof}
 
For sake of completeness let us recall \cite[Corollary 4.2]{EleCharly2003}:

\begin{proposition}
	If $a\in M^\infty(\rdd)$ and $g_1,g_2\in M^1_v(\rd)$, $w\in\cM_v$, then $\aaf $ is bounded on $M^{p,q}_w(\rd)$ for $1\leq p,q\leq \infty$. In particular, it is bounded on $L^2(\rd)$.
\end{proposition}

\subsection{Correlation functions}\label{3}
For $\f_1,\f_2\in\lrd$, let us introduce the so-called \emph{shifted window correlation function} of the pair $(\f_1, \f_2)$:
\begin{equation}\label{CF}
\mathbb{G}_{\f_1,\f_2}(t,y)=\intrd \f_2(t-u)\overline{\f_1(y-u)}\, du.
\end{equation}
It is straightforward to show that $\mathbb{G}_{\f_1,\f_2}\in L^\infty(\rdd)$. 
Observe that the definition of $\mathbb{G}_{\f_1,\f_2}$ also works for windows $\f_1,\f_2$ belonging to function/distributions spaces other than $L^2(\rd)$ (see Proposition \ref{propwindow}). \par

We can rewrite the shifted window correlation function $\mathbb{G}_{\f_1,\f_2}$ on $\rdd$  as a \emph{time shift} of the mapping ${\cC}_{\f_1,\f_2}$ on $\rd$ defined in \eqref{tildeG}.

In fact,  a straightforward computation shows that 
\begin{equation}\label{tilde-conv}
\mathbb{G}_{\f_1,\f_2}(t,y)= {\cC}_{\f_1,\f_2}(y-t)=T_t{\cC}_{\f_1,\f_2}(y), \quad t,y\in\rd.
\end{equation}
Let us study the properties of  ${\cC}_{\f_1,\f_2}$.
\begin{proposition}\label{propwindow}
	The window correlation function ${\cC}_{\f_1,\f_2}$ enjoys the following properties.
	\begin{itemize}
		\item [$(i)$] If $\f_1,\f_2\in\cS(\rd)$, then ${\cC}_{\f_1,\f_2}\in\cS(\rd)$.
		\item [$(ii)$] If either $\f_1$ is in $\cS'(\rd)$ and $\f_2\in\cS(\rd)$ or $\f_1$ is in $\cS(\rd)$ and $\f_2\in\cS'(\rd)$  then ${\cC}_{\f_1,\f_2}\in\cC(\rd)$ with at most polynomial growth.
		\item [$(iii)$] If $\f_1\in L^p(\rd)$, $\f_2\in L^{p'}(\rd)$, with $1<p<\infty$,  $1/p+1/p'=1$, then ${\cC}_{\f_1,\f_2}\in \cC_0(\rd)$. If either $p=1$ ($p'=\infty$) or $p=\infty$ ($p'=1$) then ${\cC}_{\f_1,\f_2}\in \cC_b(\rd)$.
		The same statements hold if we replace the Lebesgue space $L^p(\rd)$ (resp. $L^{p'}(\rd)$) with the modulation space $M^p(\rd)$ (resp.  $M^{p'}(\rd)$).
		\item [$(iv)$] 
		If $\f_1\in M^{p,u}_{w_1\otimes \nu}(\Ren)$, $\f_2\in M^{q,t}_{v_1\otimes
			v_2\nu^{-1}}(\Ren)$, with $1\leq
		p,q,r,t,u,\gamma\leq\infty$ satisfying \eqref{Holderindices} and \eqref{Youngindicesrbig}, and the weights as in the assumptions of Proposition \ref{propconv}, then  ${\cC}_{\f_1,\f_2}$ is in $\in M^{r,\gamma}_w(\Ren)$,
		with  norm inequality  $$\|{\cC}_{\f_1,\f_2}\|_{M^{r,\gamma}_w}\lesssim
		\|\f_1\|_{M^{p,u}_{w_1\otimes \nu}}\|\f_2\|_{ M^{q,t}_{v_1\otimes
				v_2\nu^{-1}}}.$$
	\end{itemize}
\end{proposition}
\begin{proof}
	The proofs of items $(i)$, $(ii)$ follow by the convolution properties for the Schwartz class $\cS$, its dual $\cS'$ respectively, see, e.g., the textbooks \cite{Folland,Heil}. Item $(iii)$ is a consequence of the convolution properties for $L^p(\rd)$ spaces which can be found e.g., in \cite{Folland,Heil}. For modulation spaces $M^p$ we use the convolution properties in Proposition \ref{2.4}. \par
	$(iv)$. By assumption all the weights under consideration are even, so that $\cI \f_2\in  M^{q,t}_{v_1\otimes
		v_2\nu^{-1}}(\Ren)$ whenever $\f_2\in M^{q,t}_{v_1\otimes
		v_2\nu^{-1}}(\Ren)$. Moreover modulation spaces are closed under complex conjugation, hence the result immediately follows by applying the convolution relations in Proposition \ref{propconv}.
\end{proof}

\begin{example}\label{correlGauss} In what follows we exhibit examples of window correlation functions.\\
	(i) Consider  two $L^2$-normalized Gaussian functions $\f_1(t)=\f_2(t)= 2^{d/4}e^{-\pi t^2}$, $t\in\rd$. In this case, the window correlation function ${\cC}_{\f_1,\f_2}$ in \eqref{tildeG} is a Gaussian as well
	\begin{equation}\label{correlGauss2}
	{\cC}_{\f_1,\f_2}(t)=\cI(\f_1\ast \cI(\hat{\f_2}))(t)=2^{d/2}(e^{\pi (\cdot)^2 }\ast e^{\pi (\cdot)^2})(-t)=e^{-\frac{\pi}{2}t^2},\quad t\in\rd.
	\end{equation}
	(ii) Consider $\f_1=\chi_{[0,1]^d}$,  $\f_2(t)=1$, for every $t\in\rd$. Observe $\f_1\in  L^1(\rd)$, $\f_2 \in L^\infty(\rd)$. Then the window correlation function becomes 
	$${\cC}_{\f_1,\f_2}(t)=\f_1\ast \cI( \bar{\f_2})(-t) =\int_{[0,1]^d} dy= 1,\quad  \forall t\in\rd.$$
\end{example}
\subsection{$L^{r,\infty}$ quasi-norms of rescaled Gaussians}
\begin{lemma}
	For $r\in [1,\infty)$, $\lambda >0$ and $\f(t)=e^{-\pi t^2}$, $t\in\rd$, we consider the rescaled Gaussians $\f_\lambda(t):=e^{-\pi \lambda t^2}$. Then we have
	\begin{equation}
	\|\f_\lambda\|_{L^{r,\infty}(\rd)}=\frac{\left(\frac{d}{2r}\right)^{\frac{d}{2r}}}{\Gamma(\frac d2 +1)\lambda^{\frac{d}{2r}}} e^{-\frac{d}{2r}}.
	\end{equation}
	Hence,
	\begin{equation}
	\|\f_\lambda\|_{L^{r,\infty}(\rd)}= C(d,r) \lambda^{-\frac{d}{2r}},
	\end{equation}
	with $C(d,r)=e^{-\frac{d}{2r}} \left(\frac{d}{2r}\right)^{\frac{d}{2r}}\Gamma(\frac d2 +1)^{-1}$. 
\end{lemma}
\begin{proof}
	Observe that for $\alpha\geq 1$ we have $\{t:\, |\f_\lambda(t)|>\alpha\}=\varnothing$. For $0<\alpha<1$, $\{t:\, |\f_\lambda(t)|>\alpha\}=\{t:\, |t|<\pi^{-1/2}\lambda^{-1/2}(\log (1/\alpha))^{{1/2}}\}$. The Lebesgue measure of the set is given by
	$$A_\lambda:=\mu(\{t:\, |t|<\pi^{-1/2}\lambda^{-1/2}(\log (1/\alpha))^{{1/2}}\}) =\frac{\log (1/\alpha)^{\frac{d}{2}}}{\Gamma(\frac d2 +1)\lambda^{\frac{d}{2}}}.$$
		Now, using the definition of the quasi-norm in \eqref{lrinfty},
		\begin{align*}
		\|\f_\lambda\|_{L^{r,\infty}(\rd)}&=\sup_{\alpha>0} \alpha\mu(\{t:\, |\f_\lambda(t)|>\alpha\})^{\frac1r}\\
		&= \sup_{0<\alpha<1} \alpha A_\lambda^{\frac1r}\\
		&=\frac{1}{\Gamma(\frac d2 +1)\lambda^{\frac{d}{2r}}}  \sup_{0<\alpha<1} \alpha \left(\log(1/\alpha)\right)^{\frac {d}{2r}}.
		\end{align*}
		An easy computation shows that the function $y(\alpha):= \alpha \left(\log(1/\alpha)\right)^{\frac {d}{2r}}$ on $(0,1)$ admits the maximum point $t_M:=e^{-\frac{d}{2r}}$ and the maximum is $y(t_M)=(d/(2r))^{2/(2r)} e^{-2/(2r)}$,
		so that we obtain the claim.
\end{proof}
We observe that in the $L^{r,\infty}$ spaces the rescaled Gaussians behave like in the usual $L^r$ spaces, meaning $\|\f_\lambda\|_r\asymp \|\f_\lambda\|_{L^{r,\infty}}\asymp \lambda^{-d/(2r)}$. 
\section{Study the equality $A^{\f_1,\f_2}_{1\otimes m}=T_{{m_2}}$.}
The following issue can be viewed as the answer of the question raised in the introduction.
% or in a couple of dual Banach spaces $B\times B'$, with $\cS(\rd)\hookrightarrow B\hookrightarrow B'\hookrightarrow \cS'(\rd)$, with duality $\la B,B'\ra$ extending the $L^2$-inner product.  
\begin{theorem}\label{main}
	Fix multiplier symbols  $m,m_2\in \cS'(\rd)$ (resp. $m,m_2\in M^\infty(\rd)$) and  windows $\f_1,\f_2$ in $\cS(\rd)$ (resp. in $M^1(\rd)$). Then    the equality 
	\begin{equation}\label{loc-mult}
	A^{\f_1,\f_2}_{1\otimes m} =T_{m_2}\quad\mbox{on}\quad \cS(\rd)\,\, (\mbox{resp.} \,\, M^1(\rd))
	\end{equation}
	holds if and only if 
	\begin{equation}\label{link-tilde}
	m_2= {m} \ast\cF ^{-1}({\cC}_{\f_1,\f_2})\quad \mbox{in} \quad \cS'(\rd)\,\, (\mbox{resp.} \,\, M^\infty(\rd)).
	\end{equation}
	The same conclusions hold under the following assumptions:\\
	(i) The  symbols $m,m_2$ in $\cS(\rd)$  (resp. in $M^1(\rd)$) and the window functions $(\f_1,\f_2)$ in $\cS'(\rd)\times \cS(\rd)$ (resp. $M^\infty(\rd)\times M^1(\rd)$);\\
	(ii) The  symbols $m,m_2$ in $\cS(\rd)$ (resp. in $M^1(\rd)$) and the window functions  $(\f_1,\f_2)\in \cS(\rd)\times \cS'(\rd)$  (resp. $M^1(\rd)\times M^\infty(\rd)$).
\end{theorem}
\begin{proof}
	Assume  $m,m_2\in \cS'(\rd)$ and $(\f_1,\f_2)\in\cS(\rd)\times \cS(\rd)$. First, we show that the operators $A^{\f_1,\f_2}_{1\otimes m}$ and $T_{m_2}$ are well defined and continuous from $\cS(\rd)$ to $\cS'(\rd)$. 
	For every $f,g\in\cS(\rd)$,  the weak definition of STFT multiplier 
	\eqref{weak def} and the standard properties of the STFT
	give the result, since  $V_{\f_1}f\in \cS(\rdd)$ and $V_{\f_2}g\in\cS(\rdd)$ and the mappings $V_{\f_1}$, $V_{\f_2}$ are continuous on $\cS(\rd)$, see for example \cite[Chapter 1]{Elena-book}.
	For the Fourier multiplier we use the continuity of $\cF$ (resp.  $\cF^{-1}$) on $\cS(\rd)$ (resp. $\cS'(\rd)$) and of the product $\cS(\rd)\cdot \cS'(\rd)\hookrightarrow \cS'(\rd)$.\par 
	Writing them  as integral operators we obtain
	$$A^{\f_1,\f_2}_{1\otimes m} f (t) =\intrd K_A(t,y)f(y)dy,$$
	with kernel
	\begin{align}
	K_A(t,y)&= \intrd\intrd e^{2\pi i (t-y)\o} m(\o) \f_2(t-x)\overline{\f_1(y-x)}\, dxd\o\notag\\
	&=\hat{m}(y-t)\mathbb{G}_{\f_1,\f_2}(t,y)=T_t(\hat{m}{\cC}_{\f_1,\f_2})(y),\label{kernelKa}
	\end{align}
	and  
	\begin{equation}\label{kernelB}
	T_{m_2}f(t)=\intrd K_B(t,y)f(y)dy,
	\end{equation}
	with kernel 
	\begin{equation}\label{kernelmB} 
		K_B(t,y)=\intrd e^{2\pi i (t-y)\o} m_2(\o) d\o=\hat{m}_2(y-t)=T_t\hat{m}_2(y).
	\end{equation}
	By the Schwartz' kernel theorem the operators $A^{\f_1,\f_2}_{1\otimes m}$ and $T_{m_2}$ coincide if and only if their kernels $K_A$ and $K_B$ coincide in $\cS'(\rdd)$. Equating  the kernels we obtain \eqref{link}.  \par
	Consider now  case $(i)$:  $m,m_2\in\cS(\rd)$ and $(\f_1,\f_2)\in\cS'(\rd)\times \cS(\rd)$. We use similar  arguments as above, observing  that
	the STFT $V_{\f_1}f\in \cS'(\rdd)$ for every $f\in\cS(\rd)$ (cf.  \cite[Chapter 1]{Elena-book}). The case $(ii)$ is analogous and left to the reader.\par
	Second, assume $m,m_2\in M^\infty(\rd)$, $\f_1,\f_2\in M^1(\rd)$. 	We  use the same arguments as in the first step, simply replacing  $\cS$ with $M^1$ and its dual $\cS'$ with $(M^1)'= M^\infty$. Hence, we obtain that  $T_{m_2}$ and the STFT multiplier  $A^{\f_1,\f_2}_{1\otimes m}$ are well-defined linear and bounded operators from  $M^1(\rd)$ into $M^\infty(\rd)$. 
	Rewriting them as integral operators and using the kernel theorem in the framework of modulation spaces \cite{fei0,Feichtingerkerneltheorem1997}
	we come up to the result. The cases: (i)  $m,m_2\in M^1(\rd)$, $\f_1\in M^1(\rd)$ $\f_2\in M^\infty(\rd)$, (ii)  $m,m_2\in M^1(\rd)$, $\f_1\in M^\infty(\rd)$ $\f_2\in M^1(\rd)$ are similar.  
\end{proof}

In this case the symbol $m$ of the STFT multiplier is \emph{smoothed} by the convolution with the \ft\, of  the window correlation function ${\cC}_{\f_1,\f_2}$ and the result is a multiplier symbol $m_2$ of $T_{m_2}$  smoother than $m$. For example, if you consider $m\in M^\infty(\rd)$, $\f_1,\f_2  \in M^{1}(\rd)$, as explained in  Proposition \ref{propwindow} $(iv)$, then we have
$$m_2=m\ast   \cF ^{-1}({\cC}_{\f_1,\f_2})\in M^\infty(\rd)\ast \cF^{-1} M^{1}(\rd).$$
Using the convolution property in Proposition \ref{propconv}
\begin{equation}\label{e1}
m_2\in M^\infty(\rd)\ast \cF^{-1} M^{1}(\rd)=M^\infty(\rd)\ast  M^{1}(\rd)\subset  M^{\infty,1}(\rd)\subset \cC_b(\rd)
\end{equation}
and  we infer that the multiplier  symbol $m_2$ belongs to $\cC_b(\rd)$. Then one can play with the convolution properties for modulation  (and other function) spaces to obtain a Fourier multipliers' symbol $m_2$ in different function spaces.

For applications it is often useful to consider windows $\f_1,\f_2\in L^2(\rd)$ and multiplier $m\in L^\infty(\rd)$. In this case the multiplier $m_2$ enjoys the smoothing below.
\begin{lemma}\label{finestre l2}
	Assume $\f_1,\f_2\in L^2(\rd)$, $m\in L^\infty(\rd)$. Then $m_2$ as in \eqref{link-tilde} belongs to $ \cC_b(\rd)$.
\end{lemma}
\begin{proof}
	For $\f_1,\f_2\in L^2(\rd)$, the window correlation function satisfies  $\cF ^{-1}{\cC}_{\f_1,\f_2} \in L^1(\rd)$, since $\cI \f_2,\bar{\f_1}\in L^2(\rd)$ and 
	$$\cF ^{-1}({\cC}_{\f_1,\f_2})\in \cF^{-1}(L^2(\rd)\ast L^2(\rd))=\cF ^{-1}L^2(\rd)\cdot \cF ^{-1}L^2(\rd)= L^2(\rd)\cdot L^2(\rd)\subset L^1(\rd).$$
	Hence, by Proposition \ref{propwindow} $(iii)$ we obtain
	$$m_2\in L^\infty(\rd)\ast L^1(\rd)\subset \cC_b(\rd),$$
as desired.
\end{proof}

\section{Study the equality $A^{\f_1,\f_2}_{1\otimes m}=T_m$}\label{Sec-4}
We first prove by rescaling arguments the  necessary  condition in Proposition \ref{hormandernec}, i.e. $1/q\leq 1/r+1/p$.
\begin{proof}[Proof of Proposition \ref{hormandernec}]
Let us choose the multiplier $m(t)=m_\lambda(t):=\f_\lambda(t)=e^{-\pi \lambda t^2}$ and the function $f(t)=\f_\lambda(t)$ as well. Observe that $\widehat{\f_\lambda}(\xi)=\lambda^{-d/2} e^{-\pi \lambda^{-1}\xi^2}$, so that we compute 
\begin{align*}
T_{m_\lambda}\f_\lambda(t)&=\lambda^{-d/2}\cF^{-1}(e^{-\pi\frac{\lambda^2+1}{\lambda}\xi^2})(t)\\
&=(\lambda^2+1)^{-d/2}e^{-\frac{\pi \lambda}{\lambda^2+1}t^2}. 
\end{align*}
The $L^q$ norm of the function above is given by
$$\|T_{m_\lambda}\f_\lambda\|_q\asymp\lambda^{-\frac{d}{2q}}(\lambda^2+1)^{-\frac{d}{q'}},$$
with $q'$ being the conjugate exponent of $q$. We have $\|\f_\lambda\|_{p}\asymp \lambda^{-d/(2p)}$. Assuming now \eqref{hormest} in our context
$$\|T_{m_\lambda}\f_\lambda\|_q\leq C\leq \|m_\lambda\|_{L^{r,\infty}} \|\f_\lambda\|_p$$
we get 
$$\lambda^{-\frac{d}{2q}}(\lambda^2+1)^{-\frac{d}{q'}}\leq C \lambda^{-\frac{d}{2r}}\lambda^{-\frac{d}{2p}}.$$
Letting $\lambda\to 0^+$ we obtain the desired estimate \eqref{suffcondanti-Wick}. 
\end{proof}
In this section we shall use the weak definition of a STFT multiplier. Namely, for $a\in\cS'(\rdd)$, $\f_1,\f_2\in\cS(\rd)$, the STFT multiplier $\gaw$ can be defined weakly as follows
\begin{equation}\label{weak def}\la \gaw f, g\ra=\la aV_{\f_1}f, V_{\f_2}g\ra=\la a,\overline{ V_{\f_1}f} V_{\f_2}g\ra,\end{equation}
where the brackets $\la\cdot,\cdot\ra$, linear in the first component and conjugate-linear in the second one,  denote the duality between $\cS'$ and $\cS$ (or any other suitable pair of dual spaces).\par
For any symbol $a\phas=(1\otimes m)\phas=m(\o)$, $x,\o\in\rd$, the STFT multiplier $A^{\f_1,\f_2}_{1\otimes m}$ can be formally re-written in terms of the related correlation function. Assume for simplicity that the windows  $\f_1,\f_2$ and multiplier $m=m(\o)$  are in $\cS(\rd)$. We start with $f\in\cS(\rd)$; for every fixed $t\in\rd$, the integrals below are absolutely convergent  and we are allowed to use Fubini's Theorem. Moreover, it is straightforward to see that $A^{\f_1,\f_2}_{1\otimes m} f\in\cS(\rd)$. Simple computations give
\begin{align}
A^{\f_1,\f_2}_{1\otimes m} f(t)&= \intrd e^{2\pi i \o t} m(\o)\intrd f(y) e^{-2\pi i\o y } \mathbb{G}_{\f_1,\f_2}(t,y)dyd\o \notag\\
&=\intrd e^{2\pi i \o t} m(\o)\intrd f(y) e^{-2\pi i\o y } T_t {\cC}_{\f_1,\f_2}(y)dyd\o \label{e10}\\
&= \intrd e^{2\pi i \o t} m(\o)\cF({f} T_t \cC_{\f_1,\f_2})(\o)d\o. \label{e11}
%\label{loc-correl}
\end{align}

Note that, if we assume condition \eqref{G-uno}, then $ T_t \cC_{\f_1,\f_2}=1$ for every $t\in\rd$ and $A^{\f_1,\f_2}_{1\otimes m} =T_m$, as desired.

The equality \eqref{e11} suggests the introduction of a new time-frequency representation closely related  to the STFT.

\begin{definition}
	For $\f_1\in L^1(\rd),\f_2\in L^2(\rd)$, we define the two-window short-time Fourier transform  of a signal $f\in L^2(\rd)$ by 
	\begin{equation}\label{WSTFT}
\intrd e^{-2\pi i\o y }  f(y) T_t {\cC}_{\f_1,\f_2}(y)dy=\la f,M_\o T_t \overline{\cC}_{\f_1,\f_2}\ra=V_{\overline{\cC}_{\f_1,\f_2}}f(t,\o),\quad (t,\o)\in\rdd.
	\end{equation} 
\end{definition}
For $\f_1\in L^1(\rd),\f_2\in L^2(\rd)$, Young's Inequality gives  $\overline{\cC}_{\f_1,\f_2}\in  L^2(\rd)$. Thus, the integral above is absolutely convergent  for every $f\in L^2(\rd)$. The same argument applies if we replace the condition $\f_1\in L^1(\rd),\f_2\in L^2(\rd)$ with the more general one $\f_1\in L^p(\rd), \f_2\in L^q(\rd)$ such that $1/p+1/q=3/2$.

% &=\intrd  e^{2\pi i \o t} a(\o)M_{-t}(\hat{f}\ast \widehat{\widetilde{WG}})(\o)d\o\label{e12}\\
%&=\intrd   a(\o)(\hat{f}\ast \widehat{\widetilde{WG}})(\o)d\o
%%\begin{proposition}
%	For $a,\f_1,\f_2\in\cS(\rd)$, we have
%	\begin{equation}\label{e1}
%	\la \gaw f,g\ra= \la a 
%	\end{equation}
%\end{proposition}

Using \eqref{e11}, the action of the STFT multiplier $A^{\f_1,\f_2}_{1\otimes m}$  can be rewritten as
\begin{equation}\label{gaw-stft}
A^{\f_1,\f_2}_{1\otimes m} f(t)=\intrd e^{2\pi i \o t} m(\o) 
V_{\overline{\cC}_{\f_1,\f_2}}f(t,\o)d\o=\cF_2^{-1}[mV_{\overline{\cC}_{\f_1,\f_2}}f(t,\cdot)](t),\quad t\in\rd
\end{equation}
where $\cF_2^{-1}$ denotes the partial Fourier transform w.r.t. the second coordinate $\o$. The formal equality above can be made rigorous by studying the properties of the two-window short-time Fourier transform $V_{\overline{\cC}_{\f_1,\f_2}}$ and the multiplier symbol $m(\o)$.\par  
The following issue stems from Theorem \ref{main} with $m=m_2$.
% or in a couple of dual Banach spaces $B\times B'$, with $\cS(\rd)\hookrightarrow B\hookrightarrow B'\hookrightarrow \cS'(\rd)$, with duality $\la B,B'\ra$ extending the $L^2$-inner product.  
\begin{corollary}\label{maincor}
 Fix a multiplier symbol  $m\in \cS'(\rd)$ (resp. $m\in M^\infty(\rd)$) and  windows $\f_1,\f_2$ in $\cS(\rd)$ (resp. in $M^1(\rd)$). Then    the equality 
	\begin{equation}\label{loc-mult-m}
A^{\f_1,\f_2}_{1\otimes m} =T_m\quad\mbox{on}\quad \cS(\rd)\,\, (\mbox{resp.} \,\, M^1(\rd))
	\end{equation}
holds if and only if 
	\begin{equation}\label{link}
\hat{m} {\cC}_{\f_1,\f_2}=\hat{m}\quad \mbox{in} \quad \cS'(\rd)\,\, (\mbox{resp.} \,\, M^\infty(\rd)).
	\end{equation}
	The same conclusions hold under the following assumptions:\\
	(i) The  symbol $m$ in $\cS(\rd)$  (resp. in $M^1(\rd)$) and the window functions $(\f_1,\f_2)$ in $\cS'(\rd)\times \cS(\rd)$ (resp. $M^\infty(\rd)\times M^1(\rd)$).\\
	(ii) The  symbol $m$ in $\cS(\rd)$ (resp. in $M^1(\rd)$) and the window functions  $(\f_1,\f_2)\in \cS(\rd)\times \cS'(\rd)$  (resp. $M^1(\rd)\times M^\infty(\rd)$).
\end{corollary}
Straightforward consequences of the result above are the following.
\begin{corollary} Consider either $(\f_1,\f_2)\in\cS'(\rd)\times \cS(\rd)$ or $(\f_1,\f_2)\in\cS(\rd)\times \cS'(\rd)$. Then the equality \eqref{link} holds for every symbol $m\in\cS(\rd)$ if and only if condition \eqref{G-uno}
is satisfied.	
\end{corollary}
\begin{proof}
	The condition \eqref{G-uno} immediately  follows  if we take $m(\o)=e^{-\pi\o^2}\in\cS(\rd)$ in the equality  \eqref{link}.
\end{proof}
%The further issue shows that it is not possible to write a STFT multiplier as a Fourier multiplier if both the windows $\f_1$ and $\f_2$ are \emph{nice windows}, that is they belong to the Schwartz class. 
\begin{corollary} It is not possible to find $\f_1,\f_2\in\cS(\rd)$ such that the equality \eqref{loc-mult-m}  holds for every multiplier $m\in\cS'(\rd)$.
	\end{corollary}
\begin{proof}  Taking  $m(\o)=e^{-\pi\o^2}\in\cS(\rd)$ in the equality  \eqref{link} we obtain condition \eqref{G-uno}. 
Since $\f_1,\f_2\in\cS(\rd)$, by Proposition \ref{propwindow} we infer ${\cC}_{\f_1,\f_2}\in\cS(\rd)$, thus condition \eqref{G-uno} is never satisfied.
\end{proof}

Let us try to understand the condition \eqref{G-uno} better for operators having windows/symbols in modulation  spaces.

%This is not the case for  a Fourier multiplier $T_a$, which is bounded on $\lrd$ if and only if $a\in L^\infty(\rd)\subset M^\infty(\rd)$. The last inclusion is strict: recall that $M^\infty(\rd)$ contains not only functions but also distributions, simple example being the Dirac delta $\delta$. This observation makes it clear that if we choose windows $\f_1,\f_2$ in $L^2(\rd)$ then in general   the corresponding localization operator $\gaw$  with symbol in $M^\infty(\rd)$ cannot be represented as a Fourier multiplier. 

%\begin{theorem}\label{main2}
%	Consider  $\f_1,\f_2\in M^1(\rd)$,    $m\in M^{\infty}(\rd)$, $a=1\otimes m$.  Then both the Fourier multiplier $T_m=\cF^{-1}(m\cF)$ and the localization operator  $\gaw$  are well-defined linear and bounded operators from $M^1(\rd)$ into $M^\infty(\rd)$ and the equality
%	\begin{equation}\label{loc-mult2}
%	\gaw =T_m\quad\mbox{on}\quad M^1(\rd)
%	\end{equation}
%	holds if and only if 
%	\begin{equation}\label{link2}
%	\hat{m} {WG}_{\f_1,\f_2}=\hat{m}\quad \mbox{in} \quad M^{\infty}(\rd).
%	\end{equation}
%\end{theorem}

Notice that under the assumption $\f_1,\f_2\in M^1(\rd)$ the window correlation function $\cC_{\f_1,\f_2}$ is in $M^1(\rd)$ (use  Proposition \ref{propwindow} $(iv)$ or the well-known fact that $M^1$ is an algebra under convolution). As a consequence of Theorem \ref{main}, if we want condition \eqref{link} to be satisfied for every multiplier $m\in M^{\infty}(\rd)$, the window correlation function ${\cC}_{\f_1,\f_2}$ must satisfy
\begin{equation}\label{G-uno2}
{\cC}_{\f_1,\f_2}(t)=1, \quad \, t\in \rd.
\end{equation}
But this is not possible since $\cC_{\f_1,\f_2} \in M^1(\rd)\subset \cC_0(\rd)$.

To overcome this issue, we look for windows in a bigger class that could guarantee condition \eqref{G-uno2}. This requires smoother symbols.

\begin{theorem}\label{main3}
	Consider $p_1,p_2,q_1,q_2\in [1,\infty]$, with $1/p_1+1/p_2\geq 1$, $1/q_1+1/q_2\geq 1$, $\f_1\in M^{p_1,q_1}(\rd)$, $\f_2\in M^{p_2,q_2}(\rd)$, and  $m\in M^{\infty,1}(\rd)$.  Then both the Fourier multiplier $T_m$ and the STFT multiplier  $A^{\f_1,\f_2}_{1\otimes m}$  are well-defined linear and bounded operators on $\lrd$ and the equality \eqref{loc-mult-m} 
	holds on  $M^\infty(\rd)$ if and only if condition \eqref{link}
 is satisfied on $M^\infty(\rd)$.  As a consequence, if we want  \eqref{link} to be fulfilled for every symbol $m\in M^{\infty,1}(\rd)$, the window correlation function ${\cC}_{\f_1,\f_2}$ must satisfy
 \eqref{G-uno2}.
\end{theorem}
\begin{proof}
	  We start with $\f_1,\f_2\in M^{1}(\rd)\hookrightarrow M^{p,q}(\rd)$, for every $p,q\in [1,\infty]$.  
	  Notice that, if the multiplier $m\in M^{\infty,1}(\rd)$, then the localization symbol  $(1\otimes m)$ is in $M^{\infty,1}(\rdd)$, since $$(1\otimes m) \in M^{\infty,1}(\rd)\otimes M^{\infty,1}(\rd)\subset M^{\infty,1}(\rdd)$$ and we have $1\in M^{\infty,1}(\rd)$. In fact, for any fixed non-zero $g\in\cS(\rd)$, we  work out
	  $$V_g 1\phas=\cF(T_x\bar{g})(\o)=M_{-x}\hat{\bar{g}}(\o),\quad\phas\in\rdd,$$
	  so that $$\|1\|_{M^{\infty,1}(\rd)}\asymp\|V_g 1\|_{L^{\infty,1}(\rdd)}=\|\hat{\bar{g}}\|_{L^1(\rd)}=\|g\|_{\cF L^1(\rd)}<\infty.$$
	  Hence by Theorem \ref{class} the STFT multiplier $A^{\f_1,\f_2}_{1\otimes m}$ is bounded on any $M^{p,q}(\rd)$ and in particular on $\lrd$. 
	  This is also the case for  the Fourier multiplier $T_m$ with $m\in M^{\infty,1}(\rd)$, since the inclusion relation in \eqref{minfty1cont} gives in particular  $m\in L^\infty(\rd)$ and hence $T_m\in B(L^2)$ \cite{Hormander1960}.
	  Using Theorem \ref{main}, such operators coincide whenever condition
	  \eqref{link} is satisfied.
	  
	  Next, consider  $\f_1\in M^{p_1,q_1}(\rd)$, $\f_2\in M^{p_2,q_2}(\rd)$ satisfying the assumptions. We shall show that the related  kernel $K_A$ of $A^{\f_1,\f_2}_{1\otimes m}$ is in $M^\infty(\rdd)$.	  
In fact,  Proposition \ref{propwindow} $(iv)$ gives the window correlation function $\cC_{\f_1,\f_2}\in M^{\infty,1}(\rd)$.  If the multiplier  $m\in M^{\infty,1}(\rd)$, then $\hat{m}\in W(\cF L^\infty, L^1)(\rd)\hookrightarrow  W(\cF L^\infty, L^\infty)(\rd)= M^\infty(\rd)$ (cf., e.g., \cite[Chapter 2]{Elena-book}) and the multiplication relations for modulation spaces \cite[Prop. 2.4.23]{Elena-book}
$$\|\hat{m} \cC_{\f_1,\f_2}\|_{M^\infty}\lesssim \|\hat{m} \|_{M^\infty}\| \cC_{\f_1,\f_2}\|_{M^{\infty,1}}\lesssim \|m\|_{M^{\infty,1}}\|\f_1\|_{M^{p_1,q_1}}\|\f_2\|_{M^{p_2,q_2}}<\infty.$$
Hence we obtain condition \eqref{link}.
\end{proof}
Thanks to the results above, if the window functions $\f_1$ and $\f_2$ are non-smooth, they can satisfy condition \eqref{G-uno}, as in the following issue.
\begin{example} An example of window correlation functions ${\cC}_{\f_1,\f_2}$ satisfying  \eqref{G-uno}.
Consider  $\f_2=1\in M^{\infty,1}(\rd)$ and   any $\f_1\in M^{1,\infty}(\rd)$ satisfying 
\begin{equation}\label{intf2}
\intrd \f_1 (y)\, dy=1.
\end{equation}
This gives  \eqref{G-uno2}. 
In particular, observe that \eqref{intf2} is fulfilled if we consider $\f_1(t)=e^{-\pi t^2}\in\cS(\rd)\subset M^{1,\infty}(\rd)$.
Hence,  the operators $A^{\f_1,\f_2}_{1\otimes m}$ and $T_m$ coincide for every multiplier 
  $m\in M^{\infty,1}(\rd)$.
\end{example}
The realm of modulation spaces seems the only possible environment to get the equality $A^{\f_1,\f_2}_{1\otimes m}=T_m$. Also for the standard case of $L^2$-window functions the equality fails, as shown below.
\begin{theorem}\label{mainlp}
	 Consider $\f_1,\f_2\in L^2(\rd)$,  and the multiplier $m \in L^\infty(\rd)$.     Then both the Fourier multiplier $T_m$ and the STFT multiplier  $A^{\f_1,\f_2}_{1\otimes m}$ are well-defined linear and bounded operators on $\lrd$ and the equality 
	\begin{equation}\label{loc-multlp}
A^{\f_1,\f_2}_{1\otimes m} =T_m\quad\mbox{on}\quad L^2(\rd)
	\end{equation}
holds 	if and only if condition \eqref{link} is satisfied.
 As a consequence, if we want \eqref{link} to be fulfilled for every multiplier $m\in L^\infty(\rd)$, the window correlation function ${\cC}_{\f_1,\f_2}$ must satisfy
	\eqref{G-uno2}, and this is \textbf{never} the case.
\end{theorem}
\begin{proof}
The boundedness of $A^{\f_1,\f_2}_{1\otimes m}$ on $\lrd$ is  shown in \cite{Wong1999}. 
%In fact, observe that $a\phas=(1\otimes m)\phas=m(\o)\in L^\infty(\rdd)$ for $m(\o)\in L^\infty(\rd)$, so that by H\"{o}lder's inequality and the orthogonality relations for the STFT \cite[Theorem 1.2.11]{Elena-book}
%\begin{align*}
%|\la \gaw f, g\ra|&=|\la a, \overline{V_{\f_1}f}V_{\f_2}g\ra|\\
%&\leq\|a\|_{L^\infty(\rdd)}\|\overline{V_{\f_1}f}V_{\f_2}g\|_{L^1(\rdd)}\\
%&\leq\|m\|_{L^\infty(\rd)}\|V_{\f_1}f\|_{L^2(\rdd)}\|V_{\f_2}g\|_{L^{2}(\rdd)}\\
%&=\|m\|_{L^\infty(\rd)}\|{\f_1}\|_{L^2(\rd)}\|f\|_{L^2(\rd)}\|{\f_2}\|_{L^2(\rd)}\|g\|_{L^{2}(\rdd)}.
%\end{align*}
%Hence
%$$\| \gaw f\|_{L^2(\rd)}\leq \|m\|_{L^\infty(\rd)}\|\f_1\|_{L^2(\rd)}\|{\f_2}\|_{L^2(\rd)}\|f\|_{L^2(\rd)}, \quad \forall f\in\lrd.$$
%This gives the claim for $\gaw$. 
For the Fourier multiplier we recall that $T_m$ is bounded on $\lrd$ since $m$ is in $L^\infty(\rd)$ \cite{Hormander1960}. Condition \eqref{link} then follows by Theorem \ref{main}. The window correlation function $\cC_{\f_1,\f_2}$ never satisfies \eqref{link} because  $\f_1,\f_2\in L^2(\rd)$ implies $\cC_{\f_1,\f_2}\in\cC_0(\rd)$,  by Proposition \ref{propwindow} $(iii)$.
\end{proof}

A natural question is whether we can consider windows $\f_1\in M^p(\rd)$, $\f_2\in M^{p'}(\rd)$, $1\leq p,p'\leq\infty$, $1/p+1/p'=1$, and  the multiplier $m\in L^\infty(\rd)$. This is the case explained below.

\begin{proposition}
If we consider $\f_1\in M^p(\rd)$, $\f_2\in M^{p'}(\rd)$,  $1\leq p,p'\leq\infty$, $1/p+1/p'=1$, and multiplier  $m\in L^{\infty}(\rd)$, then the result in Theorem \ref{main} holds true. In particular, the equality in \eqref{link} is fulfilled if and only if Condition \eqref{G-uno2} is satisfied.
\end{proposition} 
\begin{proof} 
 The Fourier multiplier $T_m$ is obviously well-defined, linear and bounded from $\cS(\rd)$ to $\cS'(\rd)$, since $T_m$ is bounded on $\lrd$.\\
  We recall, for $f,g,\gamma\in\cS(\rd)$ with $\norm{\gamma}_{L^2}=1$, the switching property of the STFT \cite[Lemma 1.2.3]{Elena-book} and the change of window in \cite[Lemma 1.2.29]{Elena-book}. Indeed, for $(x,\o)\in\rdd$:
 \begin{equation*}
	V_fg(x,\o)=e^{-2\pi i x\cdot\o}\overline{V_gf(-x,-\o)},\qquad
	\abs{V_gf(x,\o)}\leq\left(\abs{V_\gamma f}\ast V_g\gamma\right)(x,\o).
 \end{equation*}
 For the STFT multiplier $A^{\f_1,\f_2}_{1\otimes m}$, we use its weak definition in \eqref{weak def},  H\"{o}lder's inequality, the mentioned switching property and change of window;  for every $f,g\in\cS(\rd)$,
\begin{align*}
|\la \gaw f, g\ra|&=|\la a, \overline{V_{\f_1}f}V_{\f_2}g\ra|\\
&\leq\|a\|_{L^\infty(\rdd)}\|\overline{V_{\f_1}f}V_{\f_2}g\|_{L^1(\rdd)}\\
&\leq\|m\|_{L^\infty(\rd)}\|V_{\f_1}f\|_{L^p(\rdd)}\|V_{\f_2}g\|_{L^{p'}(\rdd)}\\
&=\|m\|_{L^\infty(\rd)}\|V_{\gamma}f\ast V_{\f_1} \gamma\|_{L^p(\rdd)}\|V_{\gamma}g\ast V_{\f_2}\gamma\|_{L^{p'}(\rdd)}\\
&\leq\|m\|_{L^\infty(\rd)}\|V_{\gamma}f\|_{L^1(\rdd)}\|V_{\f_1} \gamma\|_{L^p(\rdd)}\|V_{\gamma}g\|_{L^{1}(\rdd)}\|V_{\f_2}\gamma\|_{L^{p'}(\rdd)}\\
&=\|m\|_{L^\infty(\rd)}\|f\|_{M^1(\rd)}\|V_\gamma {\f_1} \|_{L^p(\rdd)}\|g\|_{M^1(\rd)} \|V_\gamma {\f_2}\|_{L^{p'}(\rdd)}\\
&=\|m\|_{L^\infty(\rd)}\|f\|_{M^1(\rd)}\|{\f_1}\|_{M^p(\rd)}\|{\f_2}\|_{M^{p'}(\rd)}\|g\|_{M^1(\rd)}.
\end{align*}
Since $\cS(\rd)\hookrightarrow M^1(\rd)$, the estimate above gives the continuity of $A^{\f_1,\f_2}_{1\otimes m}$ from $\cS(\rd)$ into $\cS'(\rd)$. Then, arguing as in the proof of Theorem \ref{main} we obtain the claim.
\end{proof}
Considering $\f_2(t)=1$ for every $t\in\rd$, hence $\f_2\in L^\infty(\rd)\subset M^\infty(\rd)$, and any $\f_1\in M^1(\rd)$ satisfying \eqref{intf2}, we provide examples for  Condition \eqref{G-uno2} being satisfied.\par 
\section{Smoothing effects of STFT multipliers}\label{4}

Thanks to the smoothing effect of the two-window STFT we obtain boundedness results for STFT multipliers which extend the case of Fourier multipliers.
The main tool is to use the representation of $\gaw$ in \eqref{e11}, that is
$$A^{\f_1,\f_2}_{1\otimes m} f(t) = \intrd e^{2\pi i \o t} m(\o)\cF({f} T_t \cC_{\f_1,\f_2})(\o)d\o =\cF_2^{-1}[mV_{\overline{\cC}_{\f_1,\f_2}}f(t,\cdot)].$$

\begin{theorem}\label{contpq}
Assume  $1<p\leq 2\leq q<\infty$, $m\in L^{r,\infty}(\rd)$ such that condition \eqref{suffcondanti-Wick} is satisfied. Consider windows  $\f_1,\f_2\in \cS'(\rd)$ such that the correlation function satisfies
\begin{equation}\label{WCFlp}
\cC_{\f_1,\f_2}\in L^{p'}(\rd)\cap L^\infty(\rd).
\end{equation} 
%with $1/p+1/p'=1$ (conjugate exponents). 
Then the STFT  operator $A^{\f_1,\f_2}_{1\otimes m}$ is bounded from $L^p(\rd)$ into $L^q(\rd)$.
\end{theorem}
\begin{proof}
	Consider a function $f$ in $L^p(\rd)$, $p\leq 2$, then
	$$\|fT_t \cC_{\f_1,\f_2}\|_1\leq \|f\|_p\|T_t\cC_{\f_1,\f_2}\|_{p'} =\|f\|_p\|\cC_{\f_1,\f_2}\|_{p'},\quad \forall t\in\rd $$
	and  $$\|fT_t \cC_{\f_1,\f_2}\|_p\leq \|f\|_p \|T_t\cC_{\f_1,\f_2}\|_\infty\leq\|f\|_p\|\cC_{\f_1,\f_2}\|_\infty, \quad \forall t\in\rd. $$
	So that by complex interpolation, $fT_t \cC_{\f_1,\f_2}\in  L^{s}(\rd),$ for every  $1\leq s\leq p$ (hence $1/s\geq 1/p$) $\forall t\in\rd,$ 
	with $$\|fT_t \cC_{\f_1,\f_2}\|_{L^s(\rd)}\leq C\|f\|_{L^p(\rd)},$$ for a constant $C>0$ independent of $t$.\par 	
	By Theorem \ref{hormander}, if $m\in L^{r,\infty}(\rd)$, then the Fourier multiplier $$T_mf=\cF_2^{-1}[mV_{\overline{\cC}_{\f_1,\f_2}}f(t,\cdot)]=\cF_2^{-1}[m\cF_2(fT_t \cC_{\f_1,\f_2})]$$
	 acts continuously from $ L^p(\rd)\to L^q(\rd)$, with  $q\geq  2$ satisfying  the index condition in \eqref{suffcondanti-Wick}.
\end{proof}
\begin{remark}
	If $\f_1\in L^1(\rd)\cap L^2(\rd)$ and $\f_2\in L^2(\rd)$ (or vice versa) then the window correlation function satisfies  condition \eqref{WCFlp}. In fact, by Proposition \ref{propwindow} it follows that $\cC_{\f_1,\f_2}\in L^2(\rd)\cap L^\infty(\rd)\subset L^{p'}(\rd)$, for every $2\leq p'\leq \infty$.
\end{remark}
This shows the \emph{smoothing effect}  of the two-window STFT $V_{\overline{\cC}_{\f_1,\f_2}} f$. For simplicity, let us consider $f\in \lrd$. The Fourier multiplier $T_m$  takes the function $f\in\lrd$ and considers its \ft\, $\hf$ that lives in $\lrd$ by Plancherel theorem, but  we cannot infer any other further property for $f$. Instead, in the STFT multiplier $A^{\f_1,\f_2}_{1\otimes m}$   we replace $\hf$ with  the two-window STFT  $V_{\overline{\cC}_{\f_1,\f_2}}f$. Assuming the condition \eqref{WCFlp}, we obtain that $V_{\overline{\cC}_{\f_1,\f_2}}f\in \cC_b(\rdd)\cap L^2(\rdd)$ and uniformly continuous on $\rdd$ (cf. \cite[Proposition 1.2.10, Corollary 1.2.12]{Elena-book}), and this implies $V_{\overline{\cC}_{\f_1,\f_2}}f(t,\cdot)\in \cC_b(\rd)\cap L^2(\rd)$ for every fixed $t\in\rd$, so that the related multiplier $\cF^{-1}_2[mV_{\overline{\cC}_{\f_1,\f_2}}f(t,\cdot)]$ can enjoy the smoothing effect above, uniformly with respect to $t\in\rd$. 

\subsection{The anti-Wick case.} Thanks to the discussions above, we can state that an anti-Wick operator $A_{1\otimes m}^{\f,\f}$, with Gaussian windows  $\f(t)=2^{d/4}e^{-\pi t^2}$ and multiplier symbol $m\in\cS'(\rd)$,  can \emph{never} be written in the Fourier multiplier form. In fact, recalling that the window correlation function in this case  is given by  $\cC_{\f,\f}(t)= e^{-\frac{\pi}{2} t^2}$, cf. formula \eqref{correlGauss2}, we infer that condition \eqref{G-uno} is never satisfied.

Let us better understand the smoothing effects for such operators. Using the expression in \eqref{e11},   we can write
\begin{equation*}
A_{1\otimes m}^{\f,\f} f(t)=\intrd e^{2\pi i \o t} m(\o)\cF(f T_t(e^{-\frac\pi 2 (\cdot)^2}))(\o) d\o.
\end{equation*}
The anti-Wick operator in terms of the two-window STFT defined in \eqref{WSTFT} can be written as
\begin{equation*}
A_{1\otimes m}^{\f,\f} f(t)=\cF_2^{-1}[ m V_{\cC_{\f,\f}} f(t,\cdot)], \quad t\in\rd. 
\end{equation*}
  Roughly speaking, here the signal $f$  is first smoothed by multiplying with the shifted Gaussian $ T_t(e^{-\frac\pi 2 (\cdot)^2})$, that is \begin{equation}\label{gt}
  g_t(y) :=f(y) T_t(e^{-\frac\pi 2 (\cdot)^2})(y).
 \end{equation}  
Then, the multiplier $T_m$ is applied to the modified signal $g_t$. In other words,
\begin{equation}\label{AWM}
 A_{1\otimes m}^{\f,\f} f(t)=T_m(g_t)(t),\quad f\in\lrd.
\end{equation}
From the equality above, it is clear the smoothing effect of the anti-Wick operator $A^{\f,\f}_{1\otimes m}$ with respect to the Fourier multiplier $T_m$, stated in Theorem \ref{contpqwick}, that we are going to prove very easily.
\begin{proof}[Proof of Theorem \ref{contpqwick}]
 Since the window correlation function $\cC_{\f,\f}(t)= e^{-\frac{\pi}{2} t^2}$ is in $\cS(\rd)\hookrightarrow L^{p'}(\rd)\cap L^\infty(\rd)$, for any $2\leq p'<\infty$ condition in \eqref{WCFlp} is satisfied and the thesis follows by Theorem \ref{contpq}.
\end{proof}
We end up this section by showing the necessity of the indices' relation in \eqref{suffcondanti-Wick}.
\begin{theorem}\label{neccWick}
If there exists a  $C>0$ such that the anti-Wick operator satisfies
\begin{equation}\label{necwick}
 \|A^{\f,\f}_{1\otimes m}f\|_q\leq C\|m\|_{L^{r,\infty}}\|f\|_p,\quad \forall f,m\in\cS(\rd),
\end{equation}
then condition \eqref{suffcondanti-Wick} holds true.
\end{theorem}
\begin{proof}
We write condition \eqref{necwick}   for the multipliers $m_\lambda(\xi)=\f_\lambda(\xi)=e^{-\pi\lambda \xi^2}$, $\lambda>0$, and functions $f_\lambda(t)=\f_\lambda(t)$ as well. Then we compute the anti-Wick operator $A^{\f,\f}_{1\otimes m_\lambda} f_\lambda$. A tedious computation shows 
$$A^{\f,\f}_{1\otimes m_\lambda} f_\lambda(t)=c_\lambda e^{-\pi b_\lambda t^2},$$
with
$$c_\lambda:= \frac{2^{d/2}}{(6\lambda^2+4\lambda+1)^{d/2}},\quad b_\lambda:=\frac{2\lambda(6\lambda^3+10\lambda^2+9\lambda+1)}{(6\lambda^2+4\lambda+1)(2\lambda+1)^2}.$$
This yields the norm estimate 
$$\|A^{\f,\f}_{1\otimes m_\lambda} f_\lambda\|_q\asymp c_\lambda b_\lambda^{-\frac{d}{2q}}\asymp \frac{(2\lambda +1)^{\frac{d}{q}}}{\lambda^{\frac{d}{2q}}(6\lambda^2+4\lambda+1)^{\frac{d}{2q'}}(6\lambda^3+10\lambda^2+9\lambda+1)^{\frac{d}{2q}}}.
	$$
Letting $\lambda\to 0^+$ we infer the inequality in \eqref{suffcondanti-Wick}.
\end{proof}

\section{Gabor multipliers}\label{6}
The \emph{spreading representation}  of an integral operator $L$ is useful for applications (see, e.g., \cite{DT2010} or \cite[Chapter 14]{grochenig}) and formally given by:
\begin{equation*}%\label{spredrep}
Lf(t)=\intrdd \eta_L\phas(M_\o T_{x}f)(t)dxd\o,\quad f\in\cS(\rd),
\end{equation*}
where $\eta_L$, also denoted by $\eta(L)$, is called the \emph{spreading function} and is related to the kernel $K_L=K(L)$ of the operator $L$ by the following transform
\begin{equation*}%\label{spread-kernel}
\eta_L\phas=\intrd K_L(y,y-x) e^{-2\pi i \o y}\, dy.
\end{equation*}
%In addition, the spreading function $\eta_L$ and Kohn-Nirenberg symbol $\sigma_L$ of an operator are connected via the symplectic Fourier
%Transform $\cF_s$ as
%\begin{equation*}%\label{spread-Kohn}
%\eta_L\phas=\cF_s (\sigma_L)\phas=\intrdd e^{-2\pi i (y\o-x\eta)}\sigma_L(y,\eta) \,dyd\eta.
%\end{equation*}
%Whenever clear, we may suppress the lower indexes in $\eta_L$, $K_L$ and $\sigma_L$.\\
There is an exact analogue of these objects in the finite discrete case, see \cite{fekolu09}. In fact, if $L\colon\bC^N\to\bC^N$ is a linear operator then we denote its matrix representation by $K_L=K(L)$ and define its spreading function as 
\begin{equation}
\eta_L(u,v)=\sum^{N-1}_{k=0}K_L(k,k-u)e^{\frac{-2\pi i k v}{N}}.
\end{equation}
So that $L$ can be seen as a finite superposition of TF-shifts, which in the finite dimensional case are an orthonormal basis, so every matrix can be uniquely described by its spreading function:
\begin{equation*}
L=\sum^{N-1}_{k=0}\sum^{N-1}_{l=0}\eta_L(k,l)\pi(k,l).
\end{equation*}

\subsection{Finite discrete setting}
In the finite discrete setting, we will always identify $\bC^N$ with 
%periodic vectors or equivalently complex-valued functions on the cyclic group $\bZ_N$, so that $\bC^N$ stands for 
$\ell^2(\bZ_N)$. We shall denote by $\mathbf{1} \in\bC^N$ the constant function equal to $1$. From now on, we will always consider a rectangular lattice of the form 
\begin{equation}\label{Eq-rectangular-lattice}
	\sfLa=\a\bZ_N\times\b\bZ_N,\quad \a,\b\in\bN_+,\quad A\coloneqq\frac{N}{\a}\in\bN_+,\quad B\coloneqq\frac{N}{\b}\in\bN_+.
\end{equation}
%which can be thought also as $\sfLa=\a \bZ^A\times \b\bZ^B$. 
Since $\a, \b$ are divisors of $N$, $\sfLa$ is a subgroup. Therefore in this case translation and modulation operator take the form:
 \begin{equation*}
 	T_k f(t)=f(t-k),\quad M_l f(t)=e^{\frac{2\pi i l t}{N}}f(t),\qquad f\in\bC^N,\,t=0,\ldots,N-1,\,k,l\in\bZ.
 \end{equation*}
We put again $\pi(k,l)=M_l T_k$ and define the STFT of a signal $f\in\bC^N$ w.r.t. the window $g\in\bC^N$ as  the matrix in $\bC^{N\times N}$
\begin{equation*}
	V_gf(u,v)=\<f,\pi( u,v)g\>=\sum_{k=0}^{N-1}f(k)\overline{g(k-u)}e^{\frac{-2 \pi i k v}{N}}.
\end{equation*}
The Gabor system generated by a window $g\in\bC^N$ and lattice $\sfLa$ as in \eqref{Eq-rectangular-lattice} is defined as
\begin{align*}
	\cG(g,\a,\b)&\coloneqq\{\pi(k,l)g\,,\,(k,l)\in\sfLa\}\\
	&=\{\pi(\a k,\b l)g\, ,\,k=0,\ldots,A-1,\,l=0,\ldots,B-1\}.\notag
\end{align*}
A Gabor system $\cG(g,\a,\b)$ is said to be a Gabor frame for $\bC^N$ if there exist $C_1,C_2>0$ such that
\begin{equation}
	C_1\norm{f}^2_2\leq\sum_{k=0}^{A-1}\sum_{l=0}^{B-1}\abs{\<f,\pi(\a k,\b l)g\>}^2\leq C_2\norm{f}^2_2\qquad\forall f\in\bC^N,
\end{equation}
where $\<f,\pi(\a k,\b l)g\>=\sum_{u=0}^{N-1}f(u)\overline{\pi(\a k,\b l)g(u)}$ and $\norm{\cdot}_2$ is the induced norm. Since we are in finite-dimension, this is equivalent to ask that $\cG(g,\a,\b)$  spans $\bC^N$ \cite{ch16}, where the bounds $C_1,C_2$ describe the numerical properties of the transform and the quantity $\sqrt{C_2/C_1}$ is the condition number of the analysis, see \cite{BalHolNecSto2017}.\\
The discrete Fourier transform (DFT) on $\bC^N$ is the linear operator represented by the following $N\times N$ complex matrix
\begin{equation}
	(\dFou)_{k,l}\coloneqq e^{\frac{-2 \pi i k l}{N}}
\end{equation}
which inverse if given by 
\begin{equation*}
	(\dFou^{-1})_{k,l}=\frac1N e^{\frac{2 \pi i k l}{N}}.
\end{equation*}
We shall denote by $\hf$ the vector $\dFou f$, $f\in\bC^N$. Therefore, the discrete two-dimensional Fourier transform of a matrix $a\in\bC^N\times\bC^N$ and its inverse are defined as
\begin{equation}
	\sfF_2 a(u,v)\coloneqq \sum_{k=0}^{N-1}\sum_{l=0}^{N-1}a(k,l)e^{\frac{-2\pi i u k}{N}}e^{\frac{-2\pi i v l}{N}},\,\sfF^{-1}_2 a(u,v)=\frac{1}{N^2}\sum_{k=0}^{N-1}\sum_{l=0}^{N-1}a(k,l)e^{\frac{2\pi i u k}{N}}e^{\frac{2\pi i v l}{N}}.
\end{equation}
The action of $\sfF_2$ on the (pointwise) product of $a$ and $b$ in $\bC^{N\times N}\cong \bC^N\times\bC^N$ is well-known and we mention it for sake of completeness:
%\begin{equation*}
%	\sfF_2(a\cdot b)(u,v)=\frac{1}{N^2}\left(\sfF_2a\ast\sfF_2b\right)(u,v),
%\end{equation*}
\begin{equation}\label{Eq-Conv-Th-F2}
	\sfF_2(a\cdot b)=\frac{1}{N^2}\left(\sfF_2a\ast\sfF_2b\right),
\end{equation}
where the (two-dimensional discrete) convolution on the right-hand side is defined similarly to \eqref{Eq-Def-H}. 
The Kronecker delta function $\delta\in\bC^N$ is defined as
\begin{equation*}
	\delta(u)=\begin{cases}
		1\qquad&\text{for}\qquad u=0,\\
		0\qquad&\text{for}\qquad u=1,\ldots,N-1.
	\end{cases}
\end{equation*}
We recall also the following identity %which is due to the subsequent \eqref{Eq-exp-orthogonal} and the normalization chosen for the Fourier transform
:
\begin{equation*}
	\cF_N\left(\frac1N \mathbf{1}\right)(u)=\delta(u).
\end{equation*}
The so-called impulse train, or Dirac comb, will be useful in some of the subsequent computations:
\begin{align}\label{Eq-Def-impulse-train}
	\Sha_{(\a,\b)}(u,v)&\coloneqq\sum_{p=0}^{A-1}\sum_{q=1}^{B-1}\delta(u-\a p)\delta(v-\b q)\\
	&=\chi_{\a \bZ_N}(u)\cdot\chi_{\b\bZ_N}(v)\notag\\
	&=\frac{1}{\a\b}\sum_{k=0}^{N-1}\sum_{l=0}^{N-1}\delta(u-\a k)\delta(v-\b l),\notag
\end{align}
for $u,v=0,\ldots,N-1$.

For sake of the reader, we recall the \emph{Poisson summation formula} \eqref{Eq-F_N-characteristic} and its two-dimensional analogue in the following lemma, see \cite{go62} and \cite[Theorem 3.2.1]{lu09-2}. 

\begin{lemma} \label{Lem-sym-Fourier-Sha}
	Under the assumptions in \eqref{Eq-rectangular-lattice}:
	\begin{equation}\label{Eq-F_N-characteristic}
		\cF_N \chi_{\a \bZ_N}=A\chi_{A\bZ_N},
	\end{equation}
	\begin{equation}
		\sfF_2 \Sha_{(\a,\b)}=AB\,\Sha_{(A,B)}.
	\end{equation}
\end{lemma}
%\begin{proof}
%	Before showing the computations, we recall the following well-known identity for $u,v=0,\ldots,N-1$:
%	\begin{equation}\label{Eq-exp-orthogonal}
%		\sum_{k=0}^{N-1}e^{\frac{2\pi i u k}{N}}e^{\frac{-2\pi i v k}{N}}=\begin{cases}
%			N\qquad&\text{if}\qquad u=v,\\
%			0\qquad&\text{otherwise}.
%		\end{cases}
%	\end{equation}
%	Hence
%	\begin{align*}
%		\sfF_2\Sha_{(\a,\b)}(u,v)&=\frac{1}{\a\b}\sum_{k,l=0}^{N-1}\sum_{p,q=0}^{N-1}\delta(k-\a p)\delta(l-\b q)e^{\frac{-2\pi i k u}{N}} e^{\frac{-2\pi i l v}{N}}\\
%		%&=\frac{1}{\a\b}\sum_{p,q=0}^{N-1}e^{\frac{-2\pi i \a p u}{N}} e^{\frac{-2\pi i \b q v}{N}}\\
%		&=\frac{1}{\a\b}\sum_{p=0}^{N-1}e^{\frac{-2\pi i \a p u}{N}} \sum_{q=0}^{N-1}e^{\frac{-2\pi i \b q v}{N}}.
%	\end{align*}
%	Due to \eqref{Eq-exp-orthogonal}, we see that
%	\begin{equation*}
%		\sum_{p=0}^{N-1}e^{\frac{-2\pi i \a p u}{N}} =\begin{cases}
%			N\qquad&\text{if}\quad \a u=0\quad\text{mod}\,N\\
%			0\qquad&\text{otherwise}
%		\end{cases}
%	\end{equation*}
%	and $\a u=0\quad\text{mod}\,N$ is equivalent to $u=A l$ for $l=0,\ldots,\a-1$. Arguing similarly for the second summation we have
%	\begin{align*}
%		\sfF_2\Sha_{(\a,\b)}(u,v)&=\frac{1}{\a\b} N\sum_{l=0}^{\a-1}\delta(u-A l)N\sum_{k=0}^{\b-1}\delta(v- B k)\\
%		&=AB\,\Sha_{(A,B)}(u,v).
%	\end{align*}
%This concludes the proof.
%\end{proof}

The discrete symplectic Fourier transform  of a matrix $a\in\bC^{N\times N}$ is defined as
\begin{equation}
	\sfF_s a(u,v)\coloneqq\frac1N\sum_{k=0}^{N-1}\sum_{l=0}^{N-1}a(k,l)e^{\frac{2\pi i\left(lu-kv\right)}{N}}
\end{equation}
with $u,v=0,\ldots,N-1$. Hence the relation between $\sfF_2$ and $\sfF_s$ is as follows:
\begin{equation}\label{Eq-relation-Fs-F2}
	\sfF_s a(u,v)=\frac1N\sfF_2(a^T)(-u,v)=\frac1N \sfF_2a(v,-u),
\end{equation}
$a^T$ being the transpose of $a$. Recall that given two vectors $f,g\in\bC^N$, the tensor product $f\otimes g\in\bC^{N\times N}$ is the matrix 
\begin{equation*}
	f\otimes g (u,v)=f(u)g(v),\qquad u,v=0,\ldots,N-1.
\end{equation*}
We mention also that
\begin{equation}\label{Eq-Fs-property}
	\sfF_s(f\otimes\hat{g})=g\otimes\hat{f}.
\end{equation}
%A lattice $\sfLa$ as in \eqref{Eq-rectangular-lattice} yields the definition of convolution on the the lattice between matrices $a,k\in\bC^{N\times N}$:
%\begin{equation*}
%	a\ast_{\sfLa} \sigma(u,v)\coloneqq \a\b N\sum^{A-1}_{k=0}\sum^{B-1}_{l=0}a(\a k,\b l)\sigma(u-\a k,v-\b l),
%\end{equation*}
%with $u=1,\ldots,\a$ and $v=1,\ldots,\b$.
\begin{definition}
	A Fourier multiplier, or linear time invariant (LTI) filter, or convolution operator $H\colon\bC^N\to\bC^N$ is uniquely determined by the so called impulse response $h\in\bC^N$ via circular convolution
	\begin{equation}\label{Eq-Def-H}
		Hf(u)\coloneqq h\ast f(u)\coloneqq \sum^{N-1}_{k=0}h(u-k)f(k),\quad f\in\bC^N,\, u=0,\ldots,N-1,
	\end{equation}
where $u-k$ is considered modulus $N$.
\end{definition}
Clearly, $h=H\delta$ and 
\begin{equation*}
	Hf(u)=h \ast f (u)=\left(\dFou^{-1}\dFou h\ast \dFou^{-1}\dFou f\right)(u)=\dFou^{-1}\left(\hat{h}\cdot \hf \right)(u),
\end{equation*}
see \eqref{Eq-h}, $\hat{h}$ is also called \textit{frequency response}. It is straight forward to see that a LTI filter $H$ on $\bC^N$ has matrix representation
\begin{equation}\label{Eq-matrix-H}
	K_H(u,v)=h(u-v),\qquad u,v=0,\ldots, N-1.
\end{equation}
We can define the associated discrete spreading function $\eta_H\in\bC^{N\times N}$ as
\begin{equation}\label{Eq-spreading-function-H}
	\eta_H(u,v)=h\otimes\delta(u,v).
\end{equation}
 Given a rectangular lattice $\sfLa$ as in \eqref{Eq-rectangular-lattice}, windows $g_1,g_2\in\bC^N$, mask or lower symbol $a\in\bC^{N\times N}$, we define the (finite) Gabor multiplier applied to $f\in\bC^N$ as follows:
 \begin{equation}
 	\dGabor{g_1}{g_2}{a}{\sfLa}f = \sum_{k=0}^{A-1}\sum_{l=0}^{B-1}a(\a k,\b l)V_{g_1}f(\a k,\b l)\pi(\a k,\b l)g_2.
 \end{equation} 
Whenever clear, we shall write $\dGab$ in place of $\dGabor{g_1}{g_2}{a}{\sfLa}$. We mention that in the finite discrete setting Gabor multipliers coincide with STFT multipliers if we choose $\a=1=\b$. It is straightforward to obtain the matrix representation of $\dGabor{g_1}{g_2}{a}{\sfLa}$:
\begin{equation}\label{Eq-matrix-dGab}
	K(\dGabor{g_1}{g_2}{a}{\sfLa})(u,v)=\sum_{k=0}^{A-1}\sum_{l=0}^{B-1}a(\a k,\b l)\overline{g_1(v-\a k)}g_2(u-\a k)e^{\frac{2 \pi i \b l (u-v)}{N}}.
\end{equation}
Let us introduce the notation
\begin{equation}\label{Eq-sym-Fourier-a}
	\sfA\coloneqq\sfF_s a,
\end{equation}
where $a\in\bC^{N\times N}$ is the symbol of a Gabor multiplier. In \cite{DT2010} many results for the interrelation of spreading function and Gabor multiplier are shown. Here we give the related finite dimensional result, like the following:
\begin{proposition}
	The spreading function of a (finite) Gabor multiplier $\dGabor{g_1}{g_2}{a}{\sfLa}$ is given by
	\begin{equation}\label{Eq-spreading-function-G}
			\eta(\dGabor{g_1}{g_2}{a}{\sfLa})(u,v)=\frac{N}{\a \b}\sum_{l=0}^{\a -1}\sum_{k=0}^{\b-1}\sfA(u+ B k, v- A l) V_{g_1}g_2(u,v).
	\end{equation}
\end{proposition}
\begin{proof}
 	A direct computation gives
 	\begin{align}
 		 \eta(\dGabor{g_1}{g_2}{a}{\sfLa})(u,v)&=\sum_{t=0}^{N-1}K(\dGabor{g_1}{g_2}{a}{\sfLa})(t,t-u)e^{\frac{-2\pi i t v}{N}}\notag\\
 		 &=\sum_{t=0}^{N-1}\sum_{k=0}^{A-1}\sum_{l=0}^{B-1}a(\a k,\b l)\overline{g_1(t-u-\a k)}g_2(t-\a k)e^{\frac{2 \pi i \b l u}{N}}e^{\frac{-2\pi i t v}{N}}\notag\\
 		 &=\sum_{k=0}^{A-1}\sum_{l=0}^{B-1}a(\a k,\b l)e^{\frac{2 \pi i \b l u}{N}}\notag\\
 		 &\times\sum_{t=0}^{N-1}\overline{g_1(t-u-\a k)}g_2(t-\a k)e^{\frac{-2\pi i t v}{N}}\label{Eq-proof-spreading}.
 	\end{align}
 	Performing the substitution $t'=t-\a k$ in \eqref{Eq-proof-spreading} gives
 	\begin{align*}
 		\sum_{t=0}^{N-1}\overline{g_1(t-u-\a k)}g_2(t-\a k)e^{\frac{-2\pi i t v}{N}}&=\sum_{t'=0}^{N-1}g_2(t')\overline{g_1(t'-u)}e^{\frac{-2 \pi i (t'+\a k)v}{N}}\\
 		&=\sum_{t'=0}^{N-1}g_2(t')\overline{g_1(t'-u)}e^{\frac{-2 \pi i v}{N}-\frac{-2\pi i \a k v}{N}}\\
 		&=V_{g_1}g_2(u,v)e^{-\frac{-2\pi i \a k v}{N}}.
 	\end{align*}
	Hence, recalling the definition of $\Sha_{(\a,\b)}$, $\sfF_s$, $\sfF_2$, and using Lemma \ref{Lem-sym-Fourier-Sha} together with \eqref{Eq-Conv-Th-F2}, we get
	\begin{align*}
		 \eta(\dGabor{g_1}{g_2}{a}{\sfLa})(u,v)&=\sum_{k=0}^{A-1}\sum_{l=0}^{B-1}a(\a k,\b l)e^{\frac{2 \pi i \b l u}{N}}e^{-\frac{-2\pi i \a k v}{N}}V_{g_1}g_2(u,v)\\
		 &= N \sfF_s\left(a\cdot\Sha_{(\a,\b)}\right)(u,v)V_{g_1}g_2(u,v)\\
		 &=\sfF_2 \left(a^T\cdot\Sha_{(\a,\b)}^T\right)(-u,v)V_{g_1}g_2(u,v)\\
		 &=\sfF_2 \left(a^T\cdot\Sha_{(\b,\a)}\right)(-u,v)V_{g_1}g_2(u,v)\\
		 &=\frac{1}{N^2}\left(\sfF_2 a^T\ast\sfF_2\Sha_{(\b,\a)}\right)(-u,v)V_{g_1}g_2(u,v)\\
		 &=\frac{1}{N}\sum_{k=0}^{N-1}\sum_{l=0}^{N-1}\frac1N\sfF_2 a^T(-u-k,v-l)\sfF_2\Sha_{(\b,\a)}(k,l)V_{g_1}g_2(u,v)\\
		 &=\frac{1}{N}\sum_{k=0}^{N-1}\sum_{l=0}^{N-1}\sfF_s a(u+k,v-l)AB\Sha_{(B,A)}(k,l)V_{g_1}g_2(u,v)\\
		 &=\frac{AB}{N}\sum_{l=0}^{\a -1}\sum_{k=0}^{\b-1}\sfF_s a(u+ B k, v- A l) V_{g_1}g_2(u,v).
	\end{align*}
This concludes the proof.
\end{proof}

 We shall frequently denote the periodization of $\sfA$ by $\AP$:
\begin{equation}\label{Eq-periodization-sfA}
	\AP(u,v)\coloneqq \sum_{l=0}^{\al-1}\sum_{k=0}^{\b-1}\sfA(u+Bk,v-Al),
\end{equation}
the periodicity is meant in the sense that
\begin{equation}\label{Eq-periodicity-A_P}
	\AP(u,v)=\AP(u+Bk,v+Al)
\end{equation}
for $u,v=0,\ldots,N-1$ and $k=0,\ldots,\b-1$, $l=0,\ldots,\a-1$.\\ 
So that \eqref{Eq-spreading-function-G} can be written as
\begin{equation}\label{Eq-spreading-function-G-A_P}
	\eta(\dGabor{g_1}{g_2}{a}{\sfLa})(u,v)=\frac{N}{\a \b}\AP(u,v) V_{g_1}g_2(u,v).
\end{equation}
The factor $N/\a\b$ is also called \textit{redundancy}. In the finite dimensional case the interpretation of this number is straightforward, because one uses $A\cdot B$ to represent a vector in $\bR^N$. This leads to an oversampling of
\begin{equation*}
	\frac{AB}{N}=\frac{N}{\a}\frac{N}{\b}\frac1N=\frac{N}{\a\b}.
\end{equation*}
%Let us also observe that
%\begin{equation*}
%	\AP(u,v)= \sum_{l=0}^{\al-1}\sum_{k=0}^{\b-1}\sfF_s a(u+Bk,v-Al)=\sum_{l=0}^{\al-1}\sum_{k=0}^{\b-1}T_{(-Bk,Al)}\sfF_s a(u,v)
%\end{equation*}
%and 
%\begin{equation*}
%	T_{(-Bk,Al)}\sfF_s a(u,v)=\sfF_s\left(M_{(Al,Bk)}a\right)(u,v),\qquad M_{(Al,Bk)}a(u,v)=e^{\frac{2\pi i (uAl+vBk)}{N}}a(u,v).
%\end{equation*}
%Therefore 
%\begin{equation*}
%	\AP(u,v)=\sum_{l=0}^{\al-1}\sum_{k=0}^{\b-1}\sfF_s\left(M_{(Al,Bk)} a\right)(u,v)=\sfF_s\left(\sum_{l=0}^{\al-1}\sum_{k=0}^{\b-1}M_{(Al,Bk)} a\right)(u,v),
%\end{equation*}
%let us compute the argument of the symplectic Fourier transform:
%\begin{align*}
%	\sum_{l=0}^{\al-1}\sum_{k=0}^{\b-1}M_{(Al,Bk)} a(u,v)&=\sum_{l=0}^{\al-1}\sum_{k=0}^{\b-1} e^{\frac{2\pi i (uAl+vBk)}{N}}a(u,v)\\
%	&=a(u,v)\sum_{l=0}^{\a-1}e^{\frac{2\pi i uAl}{N}}\sum_{k=0}^{\b-1}e^{\frac{2 \pi i v B k}{N}}\\
%	&=a(u,v)\a \chi_{\a\bZ_N}(u)\b\chi_{\b\bZ_N}(v),
%\end{align*}
%where the last equality is due to a generalization of \eqref{Eq-exp-orthogonal}, see e.g. \cite[(6.4)]{Pfander2013}. So that we can write
%\begin{equation}%\label{Eq-periodization-sfA-2}
%	\AP(u,v)= \sum_{l=0}^{\al-1}\sum_{k=0}^{\b-1}\sfA(u+Bk,v-Al)=\a\b\,\sfF_s\left(a\cdot \Sha_{(\a,\b)}\right)(u,v).
%\end{equation}
By using the convolution theorem for $\sfF_s$, cf. \eqref{Eq-Conv-Th-F2} and  \eqref{Eq-relation-Fs-F2} and see \cite[Theorem 4.3]{fekolu09}, and Lemma \ref{Lem-sym-Fourier-Sha} we get
\begin{equation*}
	 \sfF_s\left(a\cdot \Sha_{(\a,\b)}\right)(u,v)=\frac{1}{\a\b}\sum_{l=0}^{\al-1}\sum_{k=0}^{\b-1}\sfA(u+Bk,v-Al).
\end{equation*}
Therefore 
\begin{equation}\label{Eq-periodization-sfA-2}
	\AP(u,v)=\a\b\,\sfF_s\left(a\cdot \Sha_{(\a,\b)}\right)(u,v).
\end{equation}

\begin{example} \label{sec:firstexampl0}
As example, we consider a low pass filter as it is often implemented in practice. We choose the frequency response $\hat{h} \in \bC^N$ equal to the characteristic function, which is $1$ on $[-R,R]$ and zero elsewhere. The resulting convolution operator $H$ is compared to the filter generated by a Gabor multiplier $\dGab$ with symbol $a= \mathbf{1} \otimes\hat{h} $. As analysis and synthesis window for $\dGab$ we choose the Gaussian window normalized by the factor $1/N$, which is the redundancy since we take $\a=\b=1$. Both operations are applied to a random vector $f_0$.
%, e.g. a random signal with Fourier transform $\hat{f}_0$ supported in an interval $[-c,c]$.

%Then, we also do the best approximation by a Gabor multiplier to the convolution kernel $H$ in Hilbert Schmidt sense with the cut off impulse response $h= \dFou^{-1}(\hat{h})$. We compare the result to the result of the exact LTI filter as well as to the first approximation as described above. 
A graphical comparison of the LTI filter approach and of the Gabor multiplier one is shown in Figure \ref{fig:filtered_signal}.

%\begin{figure}[ht]
%	\centering
%	\includegraphics[width=0.90\textwidth]{Filt01.png}
%	\caption{STFT of a random signal filtered in $\bC^N$, $N=480$, by a low pass filter with cut off frequency $R=80$ samples. The spectrogram of the signal filtered by LTI filter is represented on the left %is once implemented by pointwise multiplication in the frequency domain %then with a convolution kernel, its best approximation by a Gabor multiplier 
%		and by a STFT-multiplier with mask $a= \mathbf{1}\otimes\hat{h}$ is represented on the right. %The relative error is calculated in Frobenius norm between the Gabor multiplier and the finite convolution kernel.
%	}
%	\label{fig:filtered_signal}
%\end{figure}

\begin{figure}
	\includegraphics[scale=0.4]{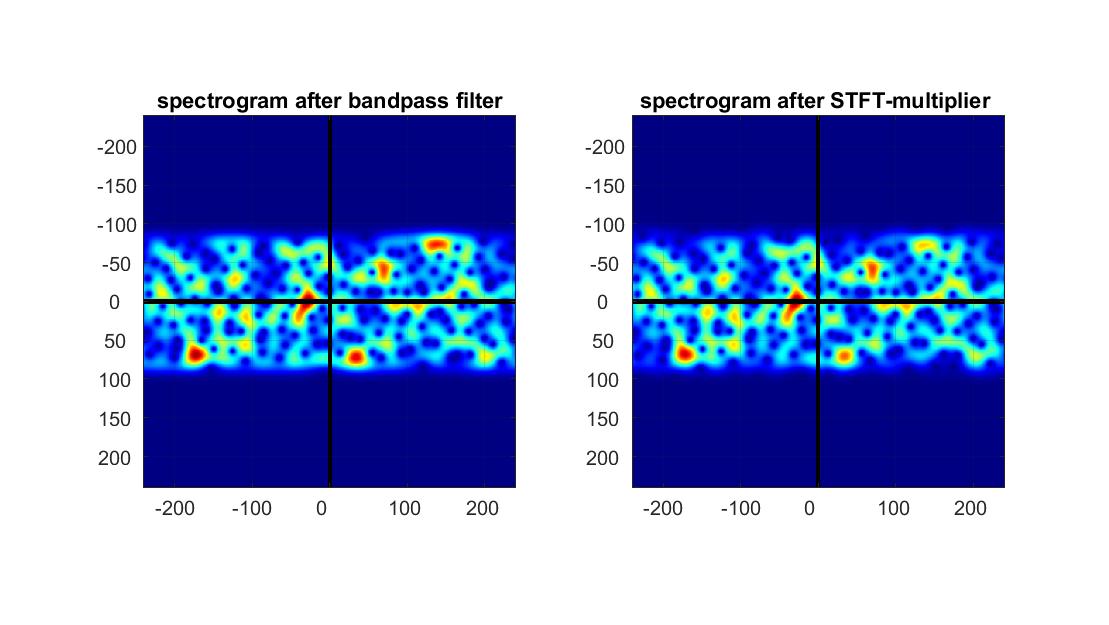}
\caption{The left part shows the (Gaussian) spectrogram of the output of an ``ideal band-pass filter'' with cut-off frequency $R=80$, applied to a random signal in $\bC^N$, $N=480$. The right hand side represents the spectrogram of the output of the corresponding STFT multiplier $\dGab$. Here we use the number of samples as normalization on time and frequency scale.}
	\label{fig:filtered_signal}
\end{figure}

\end{example}

\subsection{Representation of LTI filter as Gabor Multiplier}
\label{sec:representation}
%%%%%%%%%%%%%%%%%%%%%%%%%%
Using intuition and visual comparison as an indication that the implementation of a LTI filter by a Gabor multiplier seems to work quite well, but being aware of continuous results, we are now going to analize under which conditions it is analytically  possible to have equivalence between a LTI filter and a Gabor multiplier. We will see immediately in the first theorem that exactly the most interesting class of perfect filters with characteristic function as frequency response does not qualify as suited candidates.\\

The result below is a consequence of the general setting in Theorem \ref{mainlp}, but the estimate \eqref{Eq-estimate-1-2} is new.
\begin{theorem}
	Let $T_{m_2} \colon L^2(\bR)\rightarrow L^2(\bR)$ be a LTI filter with frequency response $m_2=\hat{h} = \chi_{\Omega}$, $\Omega\subsetneq\bR$ interval. Then  $T_{m_2}$ can never be represented exactly as Gabor multiplier with symbol $a=1\otimes m$, $m\in\ L^\infty(\bR)$, and 
	\begin{equation}\label{Eq-estimate-1-2}
		\norm{T_{m_2}-G^{g_1,g_2}_a}_{Op} \ge \frac{1}{2}  
	\end{equation}
for every  and $g_1,g_2\in L^2(\bR)$.
\end{theorem}
\begin{proof}
	The spreading function is a Banach Gelfand Triple isomorphism between  $(\cB, \cH, \cB')$ and $(S_0, L^2, S_0')(\bR^2)$, see \cite{FeiKoz1998} for notations. Therefore two operators are identical in $(\cB, \cH, \cB')$  if and only if their spreading functions are identical. 
	The integral kernel of a Fourier multiplier with symbol $m_2$ was calculated in \eqref{kernelmB} and it is related to the spreading function as follows:
	\begin{align*}
		\eta(T_{m_2})(x,\o)&=\intrd K(T_{m_2})(y,y-x)e^{-2 \pi i \o y}\,dy\\
		&=\intrd T_{y}\hat{m}_2(y-x)e^{-2\pi i \o y}\,dy\\
		&=\intrd\hat{m}_2(-x)e^{-2 \pi i \o y}\,dy\\
		&=\left(\cI\circ\cF(m_2)\otimes\delta\right)\phas\\
		&=\left(h\otimes\delta\right)\phas.
	\end{align*}
%	\textcolor{red}{The spreading function $\eta_A \in S_0'(\bR^{2d})$ of a convolution operator $A \in \cB'(L^2(\bR^d),L^2(\bR^d))$ with $A=(a*f)$ is calculated as 
%	\begin{equation}
%		\eta_A (\tau, \nu) = \int_{x \in \bR} \hat{a}(x,x-\tau) e^{- 2 \pi i x \nu}=\int_{x \in \bR} \hat{a}(\tau) e^{- 2 \pi i x \nu}=a(\tau)\delta(\nu)
%	\end{equation}}
	The spreading function of a Gabor multiplier $G^{g_1,g_2}_a$ defined trough a lattice $\La=\al\zd\times\b\zd$ is given by
	\begin{equation}
		\eta(G^{g_1,g_2}_a)(x,\o)=\cF_s(a)\phas\cdot V_{g_1}g_2\phas\eqqcolon\scrA\phas\cdot V_{g_1}g_2\phas,
	\end{equation}
	  where $\scrA=\cF_s a=\cF^{-1}(m)\otimes\cF(1)=\cF^{-1}(m)\otimes\delta$ is the $(\frac{1}{\beta}, \frac{1}{\alpha})$-periodic symplectic Fourier Transform of the symbol $a=1\otimes m$ (compare \cite{DT2010}). Therefore a Gabor multiplier is equivalent to a convolution operator if and only if 
\begin{equation} \label{eq:spreading}
 		\left(h\otimes\delta\right)\phas=\left(\cF^{-1}(m)\otimes\delta\right) V_{g_1}g_2\phas\qquad \forall \phas\in \rdd.
	\end{equation}
This  gives  
	\begin{equation}
		h(x)= \cF^{-1}(m)V_{g_1}g_2(x,0) \quad \Leftrightarrow \quad \hat{h} = m*\cF \left( V_{g_1}g_2(\cdot,0) \right).
	\end{equation}
	Let us calculate  
	\begin{align*}
		\cF \left( V_{g_1}g_2(\cdot,0) \right)(\o)&=\intrd \intrd g_2(t) \overline{g_1(t-x)}\,dt e^{-2 \pi i  \omega x}\, dx\\
		& = \intrd \intrd g_2(t) \overline{g_1(x')}e^{-2 \pi i (x-y)}\,dx' dt \\
		&= \cF(g_2)(\o)\cF^{-1}(\overline{g_1})(\o).  
	\end{align*}
	Therefore 
	\begin{equation} \label{eq:freResponse}
		\hat{h} = m \ast\cF(g_2)\cF^{-1}(\overline{g_1}).
	\end{equation}
	Since the windows $g_1,g_2$ belong to $L^2(\rd)$, we have $s(\o)\coloneqq \cF(g_2)\cF^{-1}(\overline{g_1})\in L^1(\rd)$ and the right-hand side of \eqref{eq:freResponse} is bounded and uniformly continuous. Since we are assuming $\hat{h}$ to be the characteristic function of an interval $\Omega\subsetneq\bR$, we obtained the first assertion of the thesis. \\
	About estimate \eqref{Eq-estimate-1-2} we distinguish two cases. If there is $\o_0\in\bR$ such that $\abs{s(\o_0)}=1/2$, being the image of $\hat{h}$ the set $\{0,1\}$, then 
	\begin{equation*}
		\abs{\hat{h}(\o_0)-s(\o_0)}\geq\abs{\abs{\hat{h}(\o_0)}-\abs{s(\o_0)}}\geq\frac12
	\end{equation*}
which implies 
\begin{equation*}
	\sup_{\o\in\bR}\abs{\hat{h}(\o)-s(\o)}=\norm{T_{m_2}-G^{g_1,g_2}_a}_{Op}\geq\frac12.
\end{equation*}
	 If the value $1/2$ is never attained by $\abs{s(\o)}$ the argument is identical. This concludes the proof.
\end{proof}

In the finite discrete case the problem presents itself in a similar way. In the next theorem we state the necessary conditions on the window functions in order to get perfect equivalence. It can be seen that without subsampling perfect equivalence would in theory be always possible if $\supp(h) \subseteq \supp (\overline{\cI g_1}\ast g_2)$. Numerically, we observe however the same behavior as we have in the continuous case.
\begin{theorem} \label{th:representation1}
	Le us fix a LTI filter $H\colon \bC^N\to\bC^ N$ with impulse response $h\in\bC^N$ and lattice constants $\a,\b \geq 1$.\\
	If $H$ can be written as a Gabor multiplier $\dGab$ with lattice constants $\a,\b$, for some symbol $a \in \bC^{N \times N}$ and window functions $g_1, g_2 \in \bC^N$, then the following hold for every $\forall u\in\supp(h)$:
	\begin{itemize}
		\item[1)] $V_{g_1}g_2(u, 0) = (\overline{\cI g_1} \ast g_2)(u) \neq  0$;
		\item[2)]  $V_{g_1}g_2(u+Bk, lA) = 0,$  $\forall  k= 0,\ldots,\b-1, \quad \forall l=1,\ldots,\a-1$;
		\item[3)]  $V_{g_1}g_2(u+Bk, 0) = 0,$  $\forall  k=1,\ldots,\b-1 \quad s.t. \quad (u + Bk) \notin \supp (h)$;
		\item[4)]  $V_{g_1}g_2(u+Bk, 0) =  \frac{h(u+Bk)}{h(u)}V_{g_1}g_2(u,0),$   $\forall k= 1,\ldots,\b-1\quad s.t. \quad (u + Bk) \in \supp (h)$.
	\end{itemize}
	Vice versa, if there are window functions $g_1, g_2 \in \bC^N$ fulfilling $1)$--\,$4)$, then there exists a symbol $a\in\bC^{N\times N}$ such that $H=\dGab$.
\end{theorem}
\begin{proof}
	Let us assume that $H=\dGab$ for some $a \in \bC^{N \times N},\,g_1, g_2 \in \bC^N$. Two operators are identical if and only if their spreading functions are identical. From \eqref{Eq-spreading-function-H} and \eqref{Eq-spreading-function-G-A_P}, $H=\dGab$ if and only if 
	\begin{equation} \label{eq:spread}
		(h\otimes\delta)(u,v)=\frac{N}{\a\b}\AP(u,v)V_{g_1}g_2(u,v)
	\end{equation}
	This, in turn is equivalent to
	\begin{align}
		h(u) & = N\left(\a\b\right)^{-1}\AP(u,0)V_{g_1}g_2(u,0),\qquad u=0,\ldots,N-1; \label{eq:01}\\
		0 & =  N\left(\a\b\right)^{-1}\AP(u,v)V_{g_1}g_2(u,v),\qquad u,v=0,\ldots,N-1,\quad v\neq0.\label{eq:02}
	\end{align}
	From equation \eqref{eq:01} condition 1) follows.
	Note that for $\AP(u,0) = 0$ by equation \eqref{eq:spread} we get $u \notin \supp(h)$. Hence by equation \eqref{eq:02} together with the periodicity of $\AP$ follows condition 2) follows. The periodicity of $\AP$ in the time domain together with equation \eqref{eq:01} gives condition 3).
	Finally by \eqref{eq:01} and \eqref{Eq-periodicity-A_P} we compute
	\begin{equation}
		\frac{\a\b}{N}\frac{ h(u)}{ V_{g_1}g_2(u,0)}=\AP(u,0)=\AP(u+Bk,0)=\frac{\a\b}{N}\frac{h(u+Bk)}{V_{g_1}g_2(u+Bk,0)} 
	\end{equation}
	for $k=1,...,\b-1,\,(u+kB)\in\supp(h)$, hence we get condition 4).\newline 
	On the other hand, let us consider $g_1,g_2\in\bC^N$ fulfilling conditions $1)$ --\, $4)$. 
%For $u\in\supp(h)$, $k=0,\ldots,\b-1$, and $v\in A\bZ_N$ we define
%\begin{equation}
%	\sfA_\sfP(u+Bk,v)\coloneqq\frac{\a\b}{N}\frac{ h(u)}{V_{g_1}g_2(u, 0)},
%\end{equation}
%we set $\sfA_\sfP(u,v)\coloneqq 0$ otherwise.
%	The above definition is well-posed due to 1) and 4).\\
%	Indeed suppose there are $u,z\in\supp(h)$ and $k,j\in\{0,\ldots,\b-1\}$ such that $u+Bk=z+Bj$.	Hence $\supp(h)\ni u=z+B(j-z)$ (we are omitting the operation of taking the modulus $N$) and from 4) we get
%	\begin{equation*}
%		\frac{ h(u)}{ V_{g_1}g_2(u,0)}=\frac{ h(z)}{ V_{g_1}g_2(z,0)}.
%	\end{equation*}
%    The function is periodic in the sense of \eqref{Eq-periodicity-A_P} by construction.\\
%	Equation \eqref{eq:01} is trivially satisfied. Equation \eqref{eq:02} is fulfilled if $v\notin A\bZ_N\smallsetminus\{0\}$. Let us fix $v\in A\bZ_N\smallsetminus\{0\}$ and distinguish two cases: if $u$ appearing in \eqref{eq:02} is not of the form $u'+Bk$, with $u'\in\supp(h)$ and $k\in\{0,\ldots,\b-1\}$, then $\sfA_\sfP(u,v)=0$ and we are done; if $u=u'+Bk$, with $u'\in\supp(h)$ and $k\in\{0,\ldots,\b-1\}$, then $V_{g_1}g_2(u,v)=0$ for $2)$.\\
%	By construction the defined $\AP(u,v)$ has the following form:
%	\begin{equation*}
%		\AP(u,v)=W\otimes\chi_{A\bZ_N}(u,v),
%	\end{equation*}
% 	with $W\colon\bC^N\to\bC^N$ depending  only on $u$.\\
 	Let us define for $u=0,\dots,N-1$
 	\begin{equation}
 		V(u)\coloneqq
 		\begin{cases}
 			V_{g_1}g_2(u,0)\quad&\text{if}\quad u\in\supp(h)\\
 			1\quad&\text{otherwise}
 		\end{cases}
 	\end{equation}
 and 
 \begin{equation}
 	C(u)\coloneqq\begin{cases}
 		\#\{\{u+B\bZ_N\}\cap\supp(h)\}\quad&\text{if}\quad \{u+B\bZ_N\}\cap\supp(h)\neq\varnothing\\
 		1\quad&\text{otherwise},
 	\end{cases}
 \end{equation}
we notice that $C(u+Bk)=C(u)$ for any $k=0,\ldots,\b-1$, since $u+B\bZ_N=u+Bk+B\bZ_N$.\\
Let us observe that 
\begin{align*}
	\frac{h}{C\cdot V}\ast\chi_{B\bZ_N}(u)&=\sum_{k=0}^{N-1}\frac{h(u-k)}{C(u-k)V(u-k)}\chi_{B\bZ_N}(k)=\frac{1}{C(u)}\sum_{k=0}^{N-1}\frac{h(u-k)}{V(u-k)}\chi_{B\bZ_N}(k)\\
	&%=\frac{1}{C(u)}\sum_{k=0}^{\b-1}\frac{h(u+Bk)}{V(u+Bk)}
	=\frac1C\cdot\left(\frac{h}{V}\ast\chi_{B\bZ_N}\right)(u).
\end{align*}
We define 
\begin{equation}
	\AP(u,v)\coloneqq \frac{\a\b}{N}\left(\frac{h}{C\cdot V}\ast\chi_{B\bZ_N}\right)\otimes\chi_{A\bZ_N}(u,v),
\end{equation}
which is periodic in the sense of \eqref{Eq-periodicity-A_P} since  $C(u+Bk)=C(u)$ for any $k=0,\ldots,\b-1$ and 
\begin{align*}
	\left(\frac{h}{V}\ast\chi_{B\bZ_N}\right)(u+Bk)&=\frac{1}{C(u)}\sum_{j=0}^{N-1}\frac{h(j)}{V(j)}\chi_{B\bZ_N}(u+Bk-j)\\
	&=\frac{1}{C(u)}\sum_{j=0}^{N-1}\frac{h(j)}{V(j)}\chi_{B\bZ_N-Bk}(u-j)\\
	&=\frac{1}{C(u)}\sum_{j=0}^{N-1}\frac{h(j)}{V(j)}\chi_{B\bZ_N}(u-j)\\
	&=\left(\frac{h}{V}\ast\chi_{B\bZ_N}\right)(u).
\end{align*}
In order to verify \eqref{eq:01}, fix $u\in\{0,\ldots,N-1\}$ and let us write the partition
\begin{equation*}
	\{0,\ldots,\b-1\}=S_{in}(u)\cup S_{out}(u),
\end{equation*}
where
\begin{align*}
	S_{in}(u)&\coloneqq\{k\in\{0,\ldots,\b-1\}\,|\,u+Bk\in\supp(h)\},\\
	S_{out}(u)&\coloneqq\{k\in\{0,\ldots,\b-1\}\,|\,u+Bk\notin\supp(h)\}.
\end{align*}
Therefore if  $u\in\supp(h)$ we have $0\in S_{in}(u)\neq\varnothing$, starting from the right-hand side of \eqref{eq:01}  and using 4) we get
\begin{align*}
	N(\a\b)^{-1}\AP(u,0)V_{g_1}g_2(u,0) &=\frac{1}{C(u)}\sum_{k=0}^{\b-1}\frac{h(u+Bk)}{V(u+Bk)}V_{g_1}g_2(u,0)\\
	&=\frac{1}{C(u)}\sum_{k\in S_{in}(u)}\frac{h(u+Bk)}{V_{g_1}g_2(u+Bk,0)}V_{g_1}g_2(u,0)\\
	&=\frac{1}{C(u)}\sum_{k\in S_{in}(u)}\frac{h(u)}{V_{g_1}g_2(u,0)}V_{g_1}g_2(u,0)\\
	&=\frac{1}{C(u)}C(u)h(u).
\end{align*}
If $u\notin\supp(h)$ and $S_{in}(u)=\varnothing$, then $\AP(u,0)=0$ and \eqref{eq:01} is fulfilled. If $u\notin\supp(h)$ and $S_{in}(u)\neq\varnothing$, then $u+Bj=z\in\supp(h)$ for some $j\in S_{in}(u)$. Hence we can write $u=z-Bj=z+Bs$ for a certain $s\in\{0,\ldots,\b-1\}$ and from 3) we get $V_{g_1}g_2(u,0)=V_{g_1}g_2(z+Bs,0)=0$, which guarantees \eqref{eq:01}.\\
Equation \eqref{eq:02} is fulfilled if $v\notin A\bZ_N\smallsetminus\{0\}$. Let us fix $v\in A\bZ_N\smallsetminus\{0\}$ and distinguish two cases: if $u$ appearing in \eqref{eq:02} belongs to $\supp(h)+B\bZ_N$, then $V_{g_1}g_2(u,v)=0$ due to 2) and we are done; if $u$ does not belong to $\supp(h)+B\bZ_N$, then $\AP(u,v)=0$ and \eqref{eq:02} if verified once more.\\
Eventually, in order to find a symbol $a$ which gives the function $\AP$ defined above, we use \eqref{Eq-periodization-sfA-2}:
\begin{equation}
	\a\b\sfF_s(a\cdot\Sha_{\a,\b})(u,v)=\frac{\a\b}{N}\left(\frac{h}{C\cdot V}\ast\chi_{B\bZ_N}\right)\otimes\chi_{A\bZ_N}(u,v).
\end{equation}
Being $\sfF_s^{-1}=\sfF_s$ and for \eqref{Eq-F_N-characteristic} we derive
\begin{align*}
	a(u,v)\Sha_{(\a,\b)}(u,v)&=\frac1N\sfF_s\left(\left(\frac{h}{C\cdot V}\ast\chi_{B\bZ_N}\right)\otimes\chi_{A\bZ_N}\right)(u,v)\\
	&=\frac1N A^{-1}\chi_{\a\bZ_N}(u)\cF_N\left(\frac{h}{C\cdot V}\ast\chi_{B\bZ_N}\right)(v)\\
	&=\frac{\a}{N^2}\chi_{\a\bZ_N}(u)\cF_N\left(\frac{h}{C\cdot V}\right)(v)\b\chi_{\b\bZ_N}(v)\\
	&=\frac{\a\b}{N^2}\cF_N\left(\frac{h}{C\cdot V}\right)(v)\Sha_{(\a,\b)}(u,v).
\end{align*} 
So that a possible choice for the symbol is
\begin{equation}
	a(u,v)=\frac{\a\b}{N^2}\left(\mathbf{1}\otimes\cF_N\left(\frac{h}{C\cdot V}\right)\right)(u,v).
\end{equation}	
This concludes the proof.
\end{proof}

\begin{figure}[ht]
	%\centering
	\leftskip-3cm
	\includegraphics[scale=0.2]{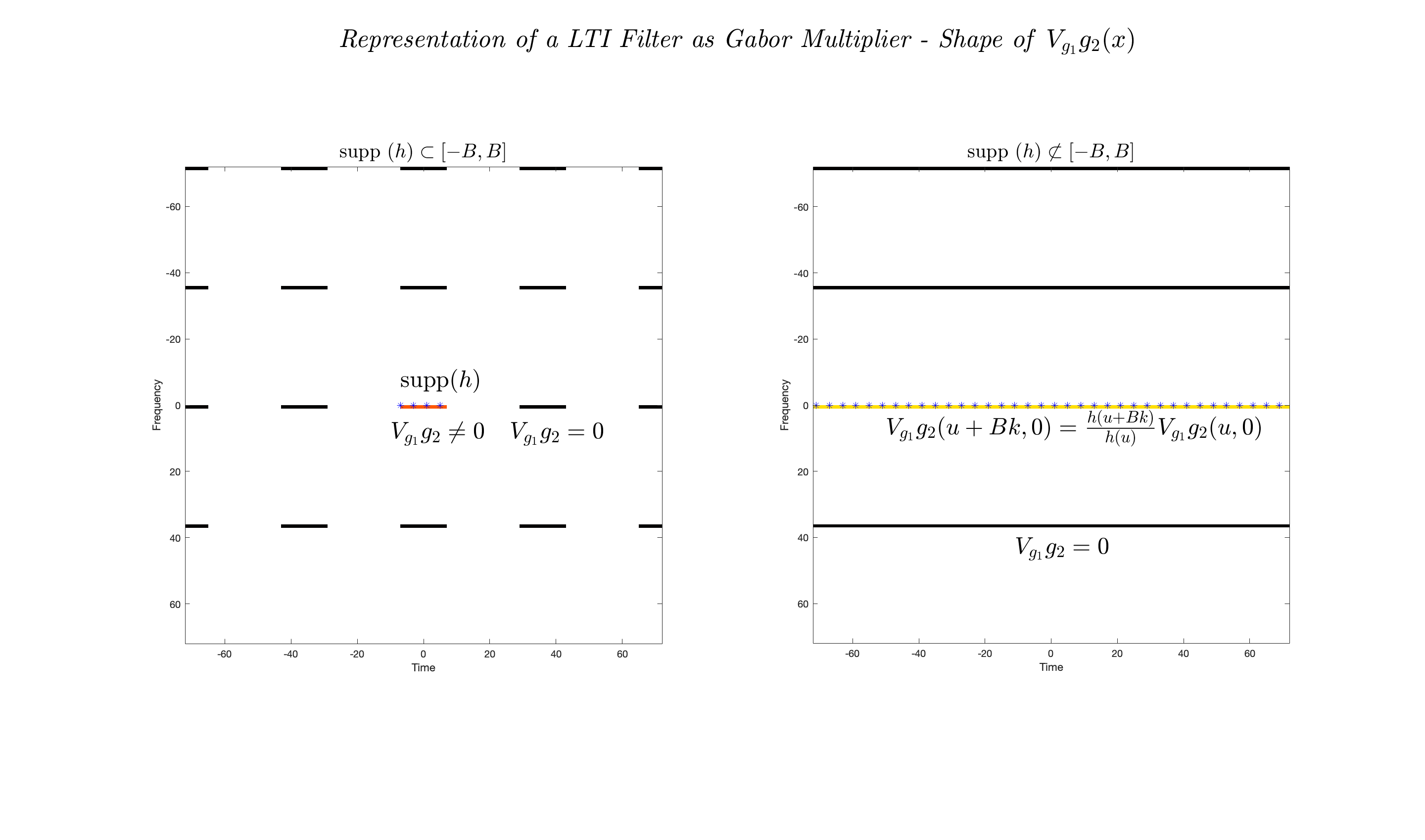}%width=0.9\textwidth
	\caption{This figure gives a visual outline of Theorem \ref{th:representation1} on the representation of a LTI filter by a Gabor multiplier. The conditions on the support of $V_{g_1}g_2$ are shown once for $\supp(h) \subset [-B,B]$ and once for $\supp(h) \not \subset [-B,B]$. Black lines indicate the regions where $V_{g_1}g_2$ has to be zero.}
	\label{fig:SkizzeSuppVGG}
\end{figure}

\begin{remark}
	Theorem \ref{th:representation1} can be seen as a special result on the reproducing property, compare \cite{hu92-2} or \cite{stta05} equation (4):
	\begin{equation*}
		f(t)\equiv\sqrt{2\pi}T\sum_{k=-\infty}^{\infty}f(kT)\psi(t-kT).
	\end{equation*}
	 If $\supp(h) \subseteq (-B,B)$ we would get perfect reproduction for $V_{g_1}g_2(u,0) = 1$ on $u\in \supp(h)$ and $V_{g_1}g_2(u,v) = 0$ outside the fundamental region of the adjoint lattice for $(u,v) \notin (-B,B) \times (-A,A)$. If $\supp(h) \subset (-B,B)$, the region $X$ with $\supp(h) \subset X \subset (-B,B)$ introduces the freedom to choose $(\overline{\cI g_1}*g_2)(u)$ having smooth decay on $X$. As for an LTI filter $\eta_H(u,v)=0 \quad \forall \, v\neq 0 $, see \eqref{Eq-spreading-function-H}, we have this freedom in the frequency domain for $Y := \left\{(x,y): 0 < |y| < A \right\}$ irrespective of the choice of $h$.
\end{remark}

The conditions given in Theorem \ref{th:representation1} will be central for the remaining part of this section. Therefore a visual outline of them is shown in Figure \ref{fig:SkizzeSuppVGG}. The next theorem can be seen as a special case of the last result,  having no subsampling, i.e. $\a=\b=1$. %Then the necessary condition reduces to $h$ having faster decay than $(g_1*g_2)$.

\begin{theorem} \label{th:representation2}
	Consider a LTI filter $H\colon\bC^N\to\bC^N$ with impulse response $h\in\bC^N$ and $g_1,g_2\in\bC^N$ with  $(\overline{\cI g_1} \ast g_2)(u)\neq 0$ for every $u=0,\ldots,N-1$. Then the $H$ can be represented as Gabor  multiplier $\dGab$ with $\a=\b=1$ and lower symbol
	\begin{equation} \label{eq:relGabMulSTFT}
		a = \frac{1}{N^2} \left( \mathbf{1} \otimes\cF_N \left( \frac{h}{\overline{\cI g_1} \ast g_2} \right)\right).
	\end{equation}
\end{theorem}
%%%%%%%%%%%%%%%%%%%%%%
%%%%%%%%%%%%%%%%%%%%%%
\begin{proof}
	Let us observe that, since $\a=\b=1$, we have $\sfA=\AP$, see \eqref{Eq-sym-Fourier-a} and \eqref{Eq-periodization-sfA-2}. Taking $a$ as in \eqref{eq:relGabMulSTFT}, recalling $\overline{\cI g_1} \ast g_2(\cdot)=V_{g_1}g_2(\cdot,0)$ and $\cF_N(N^{-1}\mathbf{1})(v)=\delta(v)$, we compute
	\begin{align*}
		\sfA(u,v)=\sfF_s a(u,v)=\frac{1}{N} \sfF_s\left(\frac1N\mathbf{1}\otimes \cF_N \left( \frac{h(\cdot)}{V_{g_1}g_2(\cdot,0)} \right)\right)(u,v)=\frac{1}{N}\frac{h(u)}{V_{g_1}g_2(u,0)}\cdot\delta(v).
	\end{align*}
	Similarly to what done in the proof  of Theorem \ref{th:representation1}, $H$ and $\dGab$ coincide if their spreading functions do; on account of the previous computation we get
	\begin{equation*} 
		h\otimes\delta(u,v)=N\sfA(u,v)V_{g_1}g_2(u,v)=\frac{h(u)}{V_{g_1}g_2(u,0)}\delta(v)V_{g_1}g_2(u,v)
	\end{equation*}
	which is true since $V_{g_1}g_2(u,0)=\overline{\cI g_1} \ast g_2(u)\neq 0$ for every $u$. This concludes the proof.
\end{proof}

This means, given window functions, for which the convolution (up to $\cI$ and a conjugation) is non-zero on the support of the impulse response $h$, a LTI filter $H$ can always be represented exactly as Gabor multiplier $\dGab$. The error between the LTI filter and the Gabor multiplier is the error introduced through subsampling of the mask $a$. The representation is always possible if we allow for the degenerate case of $g_1=g_2=\mathbf{1}$. We should, however, keep in mind that if we want to have a meaningful parameter set for applications this is, after all, a very strong  condition on the smoothness of $\hat{h}$. Even if met, for applications, the exact representation is not too well suited due to poor calculation efficiency and bad numerical behaviour for $\overline{\cI g_1} \ast g_2$ close to zero.  \\
\par
Knowing from Theorem \ref{th:representation2} that every LTI filter  with bandlimited impulse response $h$ can be represented as Gabor multiplier, we are now turning the focus to the opposite direction,  asking whether it is clear that a Gabor multiplier having a mask constant in time is equivalent to a LTI filter.  Reading equation \eqref{eq:relGabMulSTFT} the other way round, we see implicitly that a Gabor multiplier with time invariant symbol is a convolution operator. The frequency response of this convolution operator, however, is not exactly equal to the frequency mask of the Gabor multiplier but smoothed by a convolution with the Fourier transform of the window functions. In Figure \ref{fig:STFTRepresentation} a visual representation can be found. Smooth window functions  have the advantage of preserving the edges of the frequency mask rather well at the cost of a longer time delay needed in return. Theorem \ref{th:kernelSA} formalizes this fact. 

%\begin{figure}
%	\centering
%	\includegraphics[width=0.80\textwidth]{STFTRepresentation.png}
%	\caption{The figure shows the effect of implementing a Gabor (STFT) multiplier with mask $a=\chi_{\Omega} \otimes \mathbf{1}$. The resulting operator is still a LTI operator as long as no subsampling is performed ($\a=\b=1$), but the effectively performed masking effect is different. The first image in the upper left shows the frequency profile as implemented in the STFT multiplier in red vs the frequency mask of the equivalent LTI filter. It can be easily seen that the implementation as STFT multiplier leads to a smoothing effect in the edges of the frequency filter $h= \chi_{\Omega}*\left\{\cF_N(g_2 )\cF^{-1}_N(\bar{g}_1) \right\}$. The upper right plot shows the largest singular values of the difference between the implemented STFT multplier and perfect low pass filter  $\dGabor{g_1}{g_2}{a}{\sfLa}-H$, where $\sfLa=\bZ_N\times\bZ_N$ and H has impulse response $\chi_\Omega$. In the two plots below the first four singular vectors corresponding to the four largest singular values and their cumulative spectrum are plotted. It can be seen that for signals with energy concentrated at the cut off frequency the difference becomes largest. Reducing the class of signals under consideration to bandlimited functions with bandwiths only conained in $\Omega$ would of course reduce the error, but does not make sens in practice as then filtering would not be necessary.}
%	\label{fig:STFTRepresentation}
%\end{figure}

\begin{figure} [ht]

\leftskip-0.5cm
	\includegraphics[scale=0.5]{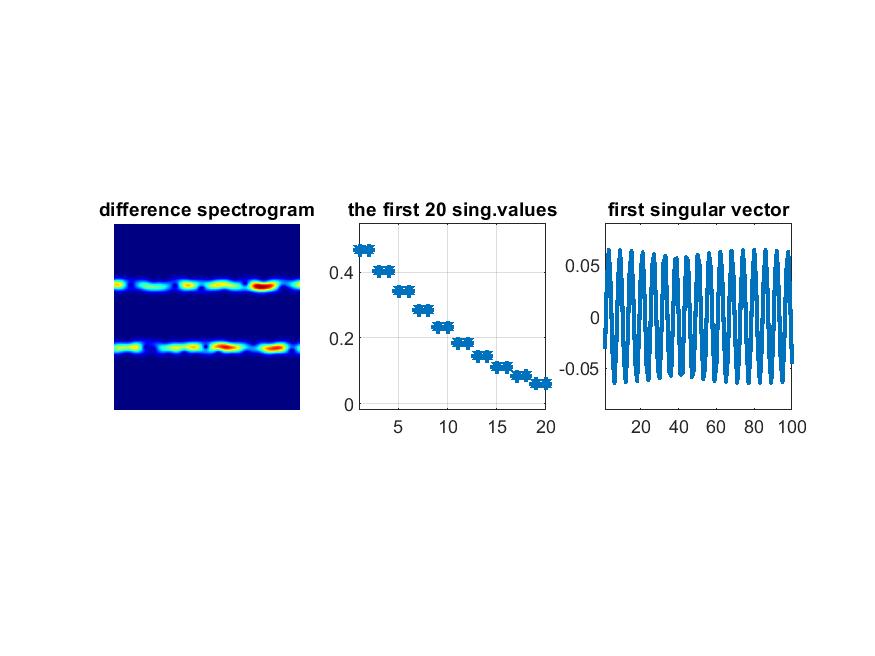}

	\caption{The figure shows the effect of implementing a Gabor (STFT) multiplier with mask $a= \mathbf{1}\otimes \chi_{\Omega} $, $\Omega=[-R,R]$, with $R=80$ and $N=480$. The resulting operator is still an LTI operator as long as no subsampling is performed ($\a=\b=1$), but now looking at the difference of the spectrograms given in the first plot, which is strongly concentrated around the cut-off frequency.
		 %The first plot represents the %spectrogram of $\dGabor{g_1}{g_2}{a}{\sfLa}-H$, where $\sfLa=\bZ_N\times\bZ_N$ %and $H$ has impulse response $\chi_\Omega$.
		The central plot shows the $20$ largest singular values of the difference between the implemented STFT multiplier and the perfect low pass filter. In the last plot, we show only a segment of the first singular vector of the difference, to demonstrate  the high regular oscillations.}
	\label{fig:STFTRepresentation}

\end{figure}

\begin{theorem}\label{th:kernelSA}
	Consider a Gabor multiplier $\dGab$ with no time subsampling, i.e. $\a=1$, windows $g_1, g_2 \in \bC^N$ with $g_1$ symmetric and symbol 
	\begin{equation}
		a =  \mathbf{1} \otimes \hat{h} 
	\end{equation}
for some $\hat{h}\in\bC^N$. Then, it is also a LTI filter with impulse response
\begin{equation}\label{Eq-imp-resp-GabMul}
	\frac1\b\sum_{k=0}^{\b-1}h(\cdot+Bk) (\overline{g_1}\ast g_2)(\cdot).
\end{equation}
\end{theorem} 
\begin{proof}
	We start from the kernel representation of the Gabor multiplier \eqref{Eq-matrix-dGab} with $\a=1$
	\begin{align}
	K(\dGab)(u,v)&=\sum_{k=0}^{N-1}\sum_{l=0}^{B-1}a( k,\b l)\overline{g_1(v- k)}g_2(u- k)e^{\frac{2 \pi i \b l (u-v)}{N}}\notag\\
		&=   \sum_{k=0}^{N-1} \sum_{l=0}^{B-1} \hat{h}(\b l) \overline{g_1(v-  k)}  g_2(u- k)e^{\frac{2\pi i \b l(u-v)}{N}} \notag  \\
		&= \sum_{l=0}^{B-1} \hat{h}(\b l) e^{\frac{2\pi i \b l(u-v)}{N}} \sum_{k=0}^{N-1}\overline{g_1(v- k)}  g_2(u- k). \label{eq:summation}
	\end{align}
	Fixing $v\in\{0,\ldots,N-1\}$, performing the change of variable $t=v-k$ and using the symmetry of $g_1$, we write the second factor as 
	\begin{align*}
		\sum_{k=0}^{N-1}\overline{g_1(v- k)}  g_2(u-k) &=\sum_{t=0}^{N-1}\overline{g_1(t)}  g_2(u-v+t)\\
		& =\sum_{t=0}^{N-1}\overline{g_1(t)}  g_2(u-v-t)\\
		&=(\overline{g_1}\ast g_2)(u-v).
	\end{align*}
 For the first factor in \eqref{eq:summation}, using \eqref{Eq-F_N-characteristic}:
 	\begin{align*}
 		\sum_{l=0}^{B-1} \hat{h}(\b l) e^{\frac{2\pi i \b l(u-v)}{N}}&=\cF^{-1}_N(\hat{h}\cdot\chi_{\b\bZ_N})(u-v)\\
 		&=\left(\cF^{-1}_N\hat{h}\ast\cF_N^{-1}\chi_{\b\bZ_N}\right)(u-v)\\
 		%&=\sum_{k=0}^{N-1}h(u-v-k)\frac1N\cF_N\chi_{\b\bZ_N}(-k)\\
 		&=\sum_{k=0}^{N-1}h(u-v-k)\frac{B}{N}\chi_{B\bZ_N}(-k)\\
 		&=\frac1\b\sum_{k=0}^{\b-1}h(u-v+Bk).
 	\end{align*}
 Eventually we get
	\begin{equation}\label{eq:LTIGMImpulseResponse}
		K(\dGab)(u,v)= \frac1\b\sum_{k=0}^{\b-1}h(u-v+Bk) (\overline{g_1}\ast g_2)(u-v)
	\end{equation}
	and the result follows by \eqref{Eq-matrix-H}.
\end{proof}
We observe that the convolution in \eqref{Eq-imp-resp-GabMul} is the restriction of $V_{g_1}g_2$ to the time-axis, since we are considering a symmetric window $g_1$.\\
It is important to note that the LTI property is only valid in case of no time subsampling. In the case of a common Gabor multiplier with $\a>1$, in contrast, the second sum in equation \eqref{eq:summation} would depend on $u$ and be $\a-$periodic, explicitly:
\begin{equation*}
	\sum_{k=0}^{A-1}\overline{g_1(v- \a k)}  g_2(u-\a k). 
\end{equation*}
Therefore as soon as we have time domain subsampling of the signal, the LTI property of the operator is lost even though the mask being constant in time. 
%An example of how fast this effect becomes visible can be seen in Figure \ref{fig:TimeSubsamplingAndTV}.

%\begin{figure}[ht]
%	\centering
%	\includegraphics[width=0.70\textwidth]{TimeSubsamplingAndTV.png}
%	\caption{The graph shows the effect of time subsampling of the signal on the time invariance property of the Gabor multiplier. Therefore a Gabor multiplier with mask $a= \chi_{\Omega} \otimes \mathbf{1}$ and different subsampling rates in time  has been implemented. The signal length was chosen to be $n=144$. Then the distance between the impulse responses $| h(\tau_1,t)-h(\tau_2,t)|$  for the operator was measures and expressed as a percentage value on the $y$-axis. Mean and maximum values are plotted. For this graph the frequency sampling parameter $\b$ was set to $4$, but the results are independent of frequency subsampling. Again, differences are mainly located in the frequency cut off region.}
%	\label{fig:TimeSubsamplingAndTV}
%\end{figure}

As already mentioned, it becomes apparent that an LTI filter can be considered as a special case of  a Gabor multiplier with degenerated window functions $g_1=g_2=\mathbf{1}$. We want to put emphasis also on the interconnection between sharp frequency cut off of the filter and smoothness of the window functions corresponding to a time delay in filtering. Condition 1) of Theorem  \ref{th:representation1} requires the impulse response $h$ to have a faster decay than $\overline{\cI g_1}\ast g_2$. This means that in case we want to have a sharp cut off in the frequency filter $\hat{h}$, which corresponds to a slow decay in $h$, we have to choose a smooth window function which corresponds to a large time lag.

\section*{Acknowledgements}
The work of P. Balazs was supported by the OeAW Innovation grant FUn ("Frames and Unbounded Operators"; IF\_2019\_24\_Fun)  and the Austrian Science Fund (FWF) project P 34624 "Localized, Fusion and Tensors of Frames" (LoFT) .\\
F. Bastianoni and E. Cordero have been supported by the Gruppo Nazionale per l’Analisi Matematica, la Probabilità e le loro Applicazioni (GNAMPA) of the Istituto Nazionale di Alta Matematica (INdAM).

\end{document}